\documentclass[hidelinks,onefignum,onetabnum]{siamart220329}

\usepackage{amsmath,amsfonts,amssymb}
\usepackage{latexsym,bm,mathrsfs}
\usepackage{lipsum}
\usepackage{epstopdf,graphicx,subfigure}
\usepackage{algorithmicx,algorithm}
\usepackage{booktabs}

\newcommand{\ket}[1]{| #1 \rangle} % |u>
\newcommand{\bra}[1]{\langle #1 |} % <u|

\def\zerobf{\textbf{0}}

\def \d {\mathrm{d}}

\newcommand{\D}{\mathrm{d}}

\newcommand{\Bb}{{\boldsymbol{b}}}

\newcommand{\Bh}{{\boldsymbol{h}}}

\newcommand{\Br}{{\boldsymbol{r}}}

\newcommand{\Bu}{{\boldsymbol{u}}}
\newcommand{\Bv}{{\boldsymbol{v}}}
\newcommand{\Bw}{{\boldsymbol{w}}}
\newcommand{\Bx}{{\boldsymbol{x}}}

\newcommand{\bbC}{\mathbb{C}}

\newcommand{\bbR}{\mathbb{R}}

\newcommand{\bbZ}{\mathbb{Z}}

\newcommand{\Ci}{\mathcal{I}}

\newcommand{\Cp}{\mathcal{P}}

\newcommand{\Ct}{\mathcal{T}}
\newcommand{\Cu}{\mathcal{U}}

\newcommand{\BB}{{\boldsymbol{B}}}

\newcommand{\BE}{{\boldsymbol{E}}}

\newcommand{\BH}{{\boldsymbol{H}}}

\newcommand{\BJ}{{\boldsymbol{J}}}

\newcommand{\BL}{{\boldsymbol{L}}}

\ifpdf
\DeclareGraphicsExtensions{.eps,.pdf,.png,.jpg}
\else
\DeclareGraphicsExtensions{.eps}
\fi

\newsiamremark{remark}{Remark}
\newsiamremark{hypothesis}{Hypothesis}
\crefname{hypothesis}{Hypothesis}{Hypotheses}
\newsiamthm{claim}{Claim}
\headers{Schr\"odingerization for linear dynamical systems}{S. Jin and N. Liu and C. Ma}

\title{On Schr\"odingerization based quantum algorithms for linear dynamical systems with inhomogeneous terms}

\author{ Shi Jin \thanks{ School of Mathematical Sciences, Shanghai Jiao Tong University, 
 Institute of Natural Sciences, Shanghai Jiao Tong University, 
 Ministry of Education, Key Laboratory in Scientific and Engineering Computing, Shanghai Jiao Tong University, Shanghai 200240, China.
(\email{shijin-m@sjtu.edu.cn})}
 \and Nana Liu \thanks{Institute of Natural Sciences, Shanghai Jiao Tong University, Ministry of Education, Key Laboratory in Scientific and Engineering Computing, Shanghai Jiao Tong University,  University of Michigan-Shanghai Jiao Tong University Joint Institute, Shanghai 200240, China.
 (\email{nanaliu@sjtu.edu.cn})}
\and Chuwen Ma \thanks{School of Mathematical Sciences, Shanghai Jiao Tong University, Institute of Natural Sciences, Shanghai Jiao Tong University, Shanghai 200240, China. 
  (\email{chuwenii@lsec.cc.ac.cn}})}

\usepackage{amsopn}

\ifpdf
\hypersetup{
	pdftitle={On Schr\"odingerization based quantum algorithms for linear dynamical systems with inhomogeneous terms},
	pdfauthor={S. Jin and N. Liu and C. Ma}
}
\fi

\begin{document}
	
	\maketitle
	
	% REQUIRED
	\begin{abstract}  
		We analyze the Schr\"odingerization method for quantum simulation of a general class of non-unitary dynamics with inhomogeneous source terms. The Schr\"odingerization technique, introduced in [S. Jin, N. Liu and Y. Yu, Phys. Rev. Lett., 133 (2024), 230602], transforms any linear ordinary and partial differential equations with non-unitary dynamics into a system under unitary dynamics via a warped phase transition that maps the equations into a higher dimension, making them suitable for quantum simulation. This technique can also be applied to these equations with inhomogeneous terms modeling source or forcing terms, or boundary and interface conditions, and discrete dynamical systems such as iterative methods in numerical linear algebra, through extra equations in the system. Difficulty arises with the presence of inhomogeneous terms since they can change the stability of the original system.

		In this paper, we systematically study—both theoretically and numerically—the important issue of recovering the original variables from the Schr\"odingerized equations, even when the evolution operator contains unstable modes. We show that, even with unstable modes, one can still construct a stable scheme; however, to recover the original variable, one needs to use suitable data in the extended space. We analyze and compare both the discrete and continuous Fourier transforms used in the extended dimension and derive corresponding error estimates, which allow one to use the more appropriate transform for specific equations. We also provide a smoother initialization for the Schr\"odingerized system to gain higher-order accuracy in the extended space. We homogenize the inhomogeneous terms with a stretch transformation, making it easier to recover the original variable. Our recovery technique also provides a simple and generic framework to solve general ill-posed problems in a computationally stable way.
	\end{abstract}

	% REQUIRED
	
	% REQUIRED  

	\section{Introduction}
	Quantum computing is considered a promising candidate to overcome many limitations of 
	classical computing \cite{Divi95,Shor94,Stea98}, one of which is the curse of dimensionality. 
	It has been shown that quantum computers could potentially outperform the most powerful classical computers, with polynomial or even exponential speed-up,  for certain types of problems \cite{Ekert98,Nielssen2000}.
	
	In this paper, we focus on the general linear dynamical system as stated below 
	\begin{equation} \label{eq:ODE}
		\frac{\D}{\D t}\Bu = A(t) \Bu(t) +\Bb(t),\quad 
		\Bu(0) = \Bu_0,
	\end{equation}  
	where $\Bu$, $\Bb\in \bbC^n$, $A\in \bbC^{n\times n}$ is a time-dependent matrix.
	This system includes ordinary differential equations (ODEs) and partial differential equations (PDEs). In the continuous representation, $A(t)$ is the differential operator; in the spatially discrete version, a PDE becomes a system of ODEs.
	Classical algorithms become inefficient or even prohibitively expensive due to large dimension $n$ whereas quantum methods can possibly significantly speed up the computation.  
	Therefore, the design for quantum simulation algorithms for solving general dynamical systems --
	including both ODEs and PDEs -- is important for a wide range of applications, for example, molecular dynamics simulations, and high dimensional and multiscale PDEs.
	
	The development of quantum algorithms with up to exponential advantage excel in  unitary dynamics, in which  $A(t)=iH(t)$ with $H=H^{\dagger}$ (where $\dagger$ represents complex conjugate)  and $\Bb=\zerobf$ in \eqref{eq:ODE},
	which is also known as the Hamiltonian simulation.
    Many efficient quantum algorithms for Hamiltonian simulation have been developed in the literature (e.g., \cite{AFL21,AFL22,BACS07,BC09,BCCKS,BCCKS15,BCK15,BCSWW22,BRS14,Cam19,CMRS18,LC17,PQS11,SS21}).
	However, most dynamical systems in applications involve non-unitary dynamics, making them unsuitable for quantum simulation.
	A very recent proposal for such problems is based on Schr\"odingerization, which converts non-unitary dynamics to Schr\"odinger-type equations with unitary dynamics. This is achieved by a warped phase transformation that maps the system to one higher dimension. The original variable $\Bu$ can then be recovered via integration or pointwise evaluation in the extra-dimensional space.  
    This technique can be applied in both qubit-based \cite{JLY22b, JLY23, JLYPRL24} and continuous-variable frameworks \cite{CVPDE2023}, with the latter being suitable for analog quantum computing. 
	The method can also be extended to solve open quantum systems in a bounded domain with artificial boundary conditions \cite{JLLY22}, physical and interface conditions \cite{JLLY23}, and even discrete dynamical systems such as iterative methods in numerical linear algebra \cite{JinLiu-LA}.

	For non-Hermitian $H(t)$, the Schr\"odingerization first decomposes it into the Hermitian and anti-Hermitian  parts:
	\begin{equation}\label{H-decomp}
		A=H_1+i H_2, \qquad H_1 = (A+A^{\dagger})/2, \qquad H_2 = (A-A^{\dagger})/(2i),
	\end{equation}
	where $H_1=H_1^{\dagger}, H_2=H_2^{\dagger}$, followed by the warped phase transformation for the $H_1$ part. 
	A basic assumption for the application of the  Schr\"odingerization technique is that the eigenvalues of $H_1$ need to be negative \cite{JLY23,JLYPRL24}. This corresponds to the stability of solution to the original system \eqref{eq:ODE}. When there is an inhomogeneous term $b(t)$, one needs to enlarge the system so as to deal with just a homogeneous system. However, the new $H_1$ corresponding to the enlarged system will have positive eigenvalues--which will be proved in this paper-- violating the basic assumption of the Schr\"odingerization method.  We address this important issue in this paper. 
	
	We show that even if some of the eigenvalues of $H_1$ are positive, the Schr\"odingerization method can still be applied. However, in this case one needs to be more careful when recovering the original variable $\Bu$. Specifically, one
	can still obtain $\Bu$ by using an appropriate domain in the new extended space, which is not `distorted' by spurious right-moving waves generated by the positive eigenvalues of $H_1$. 
	
	We will also analyze both discrete and continuous Fourier transformations used in the  extended  dimension. The corresponding error estimates for the discretization will be  established. A smoothed initialization is also introduced to provide a higher order accuracy in approximations in the extended dimension (see Remark \ref{second-order}).
	Furthermore, a stretching transformation is used to lower the impact of positive eigenvalues of $H_1$,
	which makes it easier to recover the original variables.
	We also show that the stretch coefficient does not affect the error of the Schr\"odingerization discretization.
	
	The rest of the paper is organized as follows. In section~\ref{sec:review of the Schr}, we give a brief review of the Schr\"odingerization approach for general homogeneous linear ODEs and two different Fourier discretizations in the extended space for this problem.
	In section~\ref{sec:recovery from schr}, we show the conditions for correct recovery from the warped transformation, even when there exist positive eigenvalues of $H_1$, which triggers potential instability for the extended dynamical system. We also show
	the implementation protocol on quantum devices. The error estimates and complexity analysis are established in section~\ref{sec:error for schr}. In section~\ref{sec:the inhomogenous problem}, we consider the inhomogeneous problems and show that the enlarged homogeneous system will  contain unstable modes corresponding to positive eigenvalues of $H_1$.  The recovering technique introduced in section \ref{sec:recovery from schr} can be used to deal with this system.  Finally, we show the numerical tests in section~\ref{sec:numerical tests}.

	Throughout the paper, we restrict the simulation to a finite time interval $t\in [0,T]$.
	The notation $f\lesssim g$ stands for $f\leq Cg$ where $C$ is a positive constant independent of the numerical mesh size and time step.
	Moreover, we use a 0-based indexing, i.e. $j= \{0,1,\cdots,N-1\}$, or $j\in [N]$.
	Additionally, we utilize the notation $\ket{j}\in \bbC^N$ to represent a vector where the $j$-th component being $1$ and all other components are $0$.
	We denote the identity matrix and null matrix by $I$ and $O$, respectively. 
	The dimensions of these matrices should be clear from the context; otherwise, $I_N$ refers to the $N$-dimensional identity matrix.
	Without specific statements, $\|\Bb\|$ denotes the Euclidean norm defined by $\|\Bb\|=(\sum_i |b_i|^2)^{1/2}$ and $\|A\|$  denotes 
	matrix  $2-$norms defined by $\|A\| = \sup_{\|\Bx\|=1} \|A\Bx\|$.
	The vector-valued quantities are denoted by boldface symbols, such as $\BL^2(\Omega) = (L^2(\Omega))^{n}$.
	
	\section{A review of Schr\"odingerization for  general linear dynamical systems}
	\label{sec:review of the Schr}
	
	In this section, we briefly review the Schr\"odingerization approach for general linear dynamical systems.
	Defining an auxiliary vector function $\Br(t)\equiv \sum_i \ket{i}\in \bbR^{n}$ that remains constant in time, system \eqref{eq:ODE} can be rewritten as a homogeneous system
	\begin{equation} \label{eq: homo Au}
		\frac{\D }{\D t} \begin{bmatrix}
			\Bu \\
			\Br
		\end{bmatrix}=
		\begin{bmatrix}
			A &B\\
		O &O
		\end{bmatrix}\begin{bmatrix}
			\Bu \\
			\Br
		\end{bmatrix},\qquad
		\begin{bmatrix}
			\Bu(0)\\ \Br(0)
		\end{bmatrix} = \begin{bmatrix}
			\Bu_0\\ \Br_0
		\end{bmatrix},
	\end{equation}
	where $B=\text{diag}\{b_1,b_2,\cdots,b_n\}$ and $\Br_0 = \sum_{i}\ket{i}\in \bbR^n$.
	Therefore, without loss of generality, we assume $\Bb =\textbf{0}$ in \eqref{eq:ODE}. 
	Since any matrix can be decomposed into a Hermitian term and an anti-Hermitian one as in \eqref{H-decomp}, Equation \eqref{eq:ODE} can be expressed as
	\begin{equation}\label{eq:ODE1}
		\frac{\D}{\D t} \Bu = H_1 \Bu +iH_2 \Bu, \quad  \Bu(0) = \Bu_0.
	\end{equation}

     The Schr\"odingerization method introduced in  \cite{JLY22b,JLYPRL24} assumes that {\it all eigenvalues of $H_1$ are negative}. This corresponds to assuming that the original dynamical system \eqref{eq:ODE} is {\it stable}. Using the warped phase transformation $\Bw(t,p) = e^{-p}\Bu$ for $p>0$ and symmetrically extending the initial data to $p<0$, Equation \eqref{eq:ODE1} is converted to a system of linear convection equations:
	\begin{equation}\label{eq:up}
		\frac{\D}{\D t} \Bw = -H_1 \partial_p \Bw + iH_2\Bw, \quad 
		\Bw(0) = e^{-|p|}\Bu_0.
	\end{equation}
	When the eigenvalues of $H_1$ are all negative, the convection term of \eqref{eq:up} corresponds to a wave moving from the right to the left, thus one does not need to impose a boundary condition for $\Bw$ at $p=0$.  One does, however, need a boundary condition on the right hand side. Since $\Bw$ decays exponentially in $p$, one just needs to select $p=p^*$ large enough,  so $\Bw$ at this point is essentially zero  and a zero incoming boundary condition can be used at $p=p^*$. Then $\Bu$ can be recovered via
	\begin{equation}\label{eq:pre recovery}
	\Bu(t)= \int_0^{p^*} \Bw(t,p)\, dp\quad 
	\text{or} \quad 
	\Bu(t)=e^{p} w(t,p) \quad {\text{ for any}}\quad 0<p<p^* .
	\end{equation}
	
	\subsection{The discrete Fourier transform for Schr\"odingerization}
	\label{sec:dis tran}
	We denote $\sigma(H_1)$ the set of eigenvalues of $H_1$. To discretize the $p$ domain, we choose a large enough domain $p\in \Omega_p=(-\pi L,\pi L)$, $L>0$, such that  
	\begin{equation}\label{eq:require L}
		e^{- \pi L + 2\lambda_{\max}^+(H_1)T +R} \leq \epsilon,\quad 
		e^{- \pi L + \lambda_{\max}^-(H_1)T+\lambda_{\max}^+(H_1)T +R } \leq  \epsilon,
	\end{equation}
	where $\epsilon$ is the desired accuracy, $R\geq 1$ is the length of the recovery region with  $e^R = \mathcal{O}(1)$,
   % ({\color {green} when you present  a general numerical method, you cannot put a specific constant like $\log 10$, except in numerical experiments where you have to choose some specific parameters.}) 
   and 
	\begin{equation*}
		\begin{aligned}
		   \lambda_{\max}^{+}(H_1) = \max\bigg\{\sup_{0<t<T} \{|\lambda|: \lambda\in \sigma(H_1(t)), \lambda>0\},0\bigg\},\\
			\lambda_{\max}^-(H_1) = \max\bigg\{\sup_{0<t<T}\{|\lambda|: \lambda\in \sigma(H_1(t)), \lambda<0\},0\bigg\}.
		\end{aligned}
	\end{equation*}
	 We note that \eqref{eq:require L} is intended for numerical analysis and is quite mild for numerical tests. Specifically, it implies that $e^{-\pi L + \lambda_{\max}^+(H_1)T}\lesssim  \epsilon$ and  $e^{-\pi L + \lambda_{\max}^-(H_1)T}\lesssim  \epsilon$.
     Furthermore, as implied by \eqref{eq:pre recovery}, the discretization error may grow exponentially with $p$.
     To mitigate this instability, we impose the condition $e^R=\mathcal{O}(1)$ on the recovery interval.
	 Then, we set the uniform mesh size  $\triangle p = 2\pi L/N_p$ where $N_p$ is a positive even integer and the grid points are denoted by $-\pi L=p_0<\cdots<p_{N_p}=\pi L$. Define the vector $\Bw_h$  the collection of the function $\Bw$ at these grid points by
	\begin{equation*}
		\Bw_h = \sum_{k\in [N_p]}	\sum_{j\in [n]} w_{kj} \ket{k}\ket{j}.
	\end{equation*}
	where $w_{kj}$ is  denotes the approximation to $w_j(t,p_k)$.
	The $1$-D basis functions for the Fourier spectral method are usually chosen as
	\begin{equation} \label{eq:phi nu}
		\phi_l(p) = e^{i\mu_l (p+\pi L)},\qquad \mu_l =  (l-N_p/2)/L,\quad l\in [N_p]. 
	\end{equation}
	Using \eqref{eq:phi nu}, we define 
	\begin{equation}
		\Phi = (\phi_{jl})_{N_p\times N_p} = (\phi_l(p_j))_{N_p\times N_p},  \quad
		D_{p} = \text{diag}\{\mu_0,\cdots,\mu_{N_p-1}\}.
	\end{equation}
	Considering the Fourier spectral discretisation on $p$, one easily gets
	\begin{equation}
		\frac{\D}{\D t} \Bw_h = -i(P\otimes H_1) \Bw_h + i(I_{N_p}\otimes H_2 )\Bw_h.
	\end{equation} 
	Here $P$ is the matrix representation of the momentum operator $-i\partial_p$  and
	defined by $P = \Phi D_{p} \Phi^{-1}$.
	Through a change of variables $\tilde{\Bw}_h=[\Phi^{-1}\otimes I_n]\Bw_h$, one gets
	\begin{equation}\label{eq:schro_dis}
		\begin{cases}
			\frac{\D}{\D t} \tilde{\Bw}_h = -i\big(D_p\otimes H_1  - I_{N_p}\otimes H_2\big) \tilde{\Bw}_h =-iH^d\tilde{\Bw}_h,\\
			\tilde \Bw_h(0) =\big(\Phi^{-1}\otimes I_n\big)\Bw_h(0).
		\end{cases}
	\end{equation}
	At this point, a quantum algorithm for Hamiltonian simulation can be constructed for the above linear system~\eqref{eq:schro_dis}.
   	If $H_1$ and $H_2$ are usually sparse,  then the Hamiltonian $H=D_p\otimes H_1-I_{N_p}\otimes H_2$ inherits the sparsity. 
	\subsection{The continuous Fourier transform for Schr\"odingerization}
	
	In the previous section, $\partial_p \Bw(t,p)$   is first discretized in $p$ by the discrete Fourier transform. Here, we consider the continuous Fourier transform of $\Bw(p)$ and its inverse transformation defined by
	\begin{equation*}
		\hat{\Bw}:= \mathscr{F}(\Bw)(\xi) = 
		\frac{1}{2\pi}\int_{\bbR} e^{i\xi p}\Bw(p)\,\d p,\qquad 
		\Bw:=\mathscr{F}^{-1}(\hat\Bw)(p)= \int_{\bbR} e^{-i\xi p}\hat{\Bw}(\xi)\d \xi.
	\end{equation*}
	The continuous Fourier transform in $p$ for \eqref{eq:up} gives
	\begin{equation}\label{eq:schro_conti}
		\begin{aligned}
			\frac{\D}{\D t} \hat{\Bw} &= i\xi  H_1 \hat{\Bw}+iH_2\hat{\Bw} = i(\xi H_1+H_2)\hat{\Bw},\\
			\hat{\Bw}(0) &= \frac{1}{\pi (1+\xi^2)}\Bu_0.
		\end{aligned}
	\end{equation}
	 We can first consider the truncation of the $\xi$- domain to a finite interval $[-X,X]$ with 
	\begin{equation}\label{eq:require X}
		X^{-1}=\mathcal{O}(\epsilon),
	\end{equation}
	where $\epsilon$ is the desired accuracy.
	Then canonical Hamiltonian simulation methods can be applied to the following unitary dynamics
	\begin{equation}\label{eq:schro_conti_computer}
		\begin{cases}
			\frac{\D}{\D t} \check{\Bw}_h = i(D_{\xi}\otimes H_1+I_{N+1} \otimes H_2 )\check{\Bw}_h = iH^c \check{\Bw}_h,\\
			\check{\Bw}_h(0)= \bm{\xi}_h \otimes \Bu_0,
		\end{cases}
	\end{equation}
	where $D_{\xi} = \text{diag}\{\xi_0,\xi_1,\cdots,\xi_{N}\}$ and $\bm{\xi}_h = \sum\limits_{j=0}^{N} \frac{\ket{j}}{\pi(1+\xi_j^2)}$. 
	
	It is also possible to consider the continuous Fourier transform without truncation in the $\xi$ domain. This is possible when considering implementation on analog quantum devices \cite{CVPDE2023}. In this case, instead of finite-dimensional matrices or states we define infinite dimensional vectors $|w(t)\rangle \equiv \int \Bw(t,p)|p\rangle dp$, which are acted on by infinite dimensional operators $\hat{p}$ and $\hat{\xi}$, where $[\hat{p},\hat{\xi}]=iI$. These infinite dimensional states and operators have natural meaning in the context of analog quantum computation, where they can represent for example quantum states of light or superconducting modes. If we let $|p\rangle$ and $|\xi\rangle$ denote the eigenvectors of $\hat{p}$ and $\hat{\xi}$ respectively, then $\langle p|\xi\rangle=\exp(ip \xi)/\sqrt{2\pi}$. Let $\mathcal{F}$ denote the continuous Fourier transform acting on these infinite dimensional states with respect to the auxiliary variable $p$, then $\mathcal{F}|p\rangle=|\xi\rangle$. Note that this $\mathcal{F}$ operation also has a natural interpretation in quantum systems and can be implemented in its continuous form. These $\{|p\rangle\}_{p \in \mathbb{R}}$ and $|\xi\rangle_{\xi \in \mathbb{R}}$ eigenstates each form a complete eigenbasis so $\int dp |p\rangle \langle p|=I=\int d \xi |\xi \rangle \langle \xi|$. Here the quantised momentum operator $\hat{\xi}$ is also associated with the derivative $\hat{\xi} \leftrightarrow -i\partial/\partial p$ and it is straightforward to show $i\hat{\xi} |w(t)\rangle=|\partial w/\partial p\rangle$.   
	
	In this case, Eq.~\eqref{eq:schro_conti} does not transform into Eq.~\eqref{eq:schro_conti_computer}, but rather to
	\begin{align} \label{eq:CVschro}
		& \frac{d |w(t)\rangle}{dt}=i(\hat{\xi} \otimes H_1+I \otimes H_2)|w(t)\rangle, \nonumber \\
		& |w(0)\rangle=\int \frac{1}{\pi(1+\xi^2)}|\xi\rangle d\xi\, \Bu_0.
	\end{align}

	\section{Recovery of the original variables and implementations on  quantum devices }
	\label{sec:recovery from schr}
	
	The purpose of this section is to show the rigorous conditions of the recovery $\Bu$ from $\Bw$ and the implementation on a quantum device. We consider the more general case, namely here we will {\it allow the eigenvalues of $H_1$ to be non-negative}, which is the case for many applications and extensions of the Schr\"odingerization method, for example systems with inhomogeneous terms \cite{JLY22b,JLYPRL24}, boundary value and interface problems \cite{JLLY22,JLLY23},
	and iterative methods in numerical linear algebra \cite{JinLiu-LA}. 
	
	\subsection{Recovery of the original variable}

	Since $H_1$ is Hermitian, assume it has $n$ real eigenvalues, ordered as
	\begin{equation}\label{eq:eigenvalues H1}
		\lambda_1(H_1)\leq \lambda_2(H_1)\leq \cdots\lambda_{n}(H_1), \quad 
		\text{for}\; \text{all}\; t\in [0,T].
	\end{equation}
	Next we define the solution to the time-dependent Hamiltonian system \eqref{eq:schro_conti} as
	\begin{equation*}
		\hat\Bw(t) = \Cu_{t,0}(\xi)\hat\Bw(0), 
	\end{equation*}
	where the unitary operator $\Cu_{t,s}$ is defined by the time-ordering exponential, $\Ct e$, via Dyson's series \cite{Lam98},
	\begin{equation}\label{eq:def time operator}
		\Cu_{t,s} = \Ct e^{i\int_s^t H(y)\;dy}
		= I + \sum_{n=1}^{\infty}\frac{i^n}{n!} \int_s^tdt_1\cdots \int_s^t dt_n 
		\Ct[H(t_1) H(t_2)\cdots H(t_n)],
	\end{equation}
	where $\Ct [H(t_1)H(t_2)\cdots H(t_n)] = H(t_{i_1})H(t_{i_2})\cdots H(t_{i_n})$ with $t_{i_1}>t_{i_2}\geq \cdots t_{i_n}$, 
	and $\Cu_{t,s}$ satisfies
	\begin{equation*}
		\frac{\D}{\D t} \Cu_{t,s} = i H(t)\Cu_{t,s},\quad
		\Cu_{t,s} = \Cu_{t,s'}\Cu_{s',s},\quad  \Cu_{t,s}^{\dagger} = \Cu_{s,t}.
	\end{equation*}
	\begin{theorem}\label{thm:recovery u}
		Assume the eigenvalues of $H_1$ satisfy \eqref{eq:eigenvalues H1}, the solution of 
		\eqref{eq:ODE}  can be recovered by
		\begin{equation}\label{eq:recover by one point}
			\Bu = e^{p} \Bw(t,p),\quad \text{for}\;\text{any}\quad  p\geq p^{\Diamond},
		\end{equation}
		where $p^{\Diamond}\geq \lambda_{\max}^+(H_1) T$, or recovered by using the integration,
		\begin{equation}\label{eq:recover by quad}
			\Bu = e^{p}\int_{p}^{\infty} \Bw(q)\;dq,\quad \text{for}\;\text{any}\quad  p\geq p^{\Diamond}.
		\end{equation}
	\end{theorem}
	\begin{proof}
		Following the proof in \cite[Theorem 5, section 7.3]{evan16}, we conclude that the solution of \eqref{eq:up} is unique and given by 
		\begin{equation}\label{eq:w}
			\Bw(t,p) = \int_{\bbR} \hat{\Bw}(t,\xi) e^{-i\xi p}\;d\xi =
			\int_{\bbR} \frac{\Cu_{t,0}(\xi) \Bu_0}{\pi(1+\xi^2)}e^{-i\xi p}\;\d \xi.
		\end{equation}
		It is sufficient to prove that for any $p>p^{\Diamond}$, $\Bw(t,p)$ satisfies
		\begin{equation*}
			\frac{\D}{\D t} \Bw(t,p) = A(t)\Bw(t,p), \quad \Bw(0,p) = e^{-p}\Bu_0.
		\end{equation*}
		It is obvious  that the initial condition holds by letting $t=0$ in \eqref{eq:w}.
		Since $p> \lambda_{\max}^+(H_1) T$, for any $t\in[\delta, T]$, where $\delta>0$, one has 
		$\lambda(H_1)-p/t\leq 0$. According to \cite[Lemma~5]{ALL}, there holds
		\begin{equation}\label{eq:zero}
			\Cp \int_{\bbR} \frac{\Cu_{t,0}(\xi) e^{-i\xi p}}{1- i\xi} d\xi = 
			\Cp \int_{\bbR} \frac{1}{1- i\xi} \Ct e^{i\int_0^t (\xi (H_1(s)-p/t)+H_2(s))\;\d\,s}d\xi =0,
		\end{equation}
		where $\Cp$ stands for the Cauchy principal value of the integral.
		Differentiating $\Bw(t)$ with respect to t gives
		\begin{equation}\label{eq:dtw}
			\begin{aligned}
				\frac{\D}{\D t} \Bw(t,p) &=\Cp \int_{\bbR} \frac{i (\xi H_1(t)+H_2(t))
					\Cu_{t,0}(\xi )}{\pi(1+\xi^2)} 
				\Bu_0 e^{-i\xi p}\;d\xi  \\
				&=-\Cp \int_{\bbR}\frac{H_1(t) \Cu_{t,0}(\xi ) e^{-i\xi p}}{\pi(1-i\xi )}\Bu_0\;\d \xi  
				+\Cp\int_{\bbR} \frac{A(t)\Cu_{t,0}(\xi )\Bu_0}{\pi(1+\xi^2)} e^{-i\xi p}  \;\d \xi \\
				&=A(t) \Bw(t,p).
			\end{aligned}
		\end{equation}
		Considering $\delta$ can be small enough, \eqref{eq:dtw} holds for any $t\in(0,T]$.
		The proof is finished by the fact that $\Bw(t,p) = e^{-p}\Bu(t)$ for $p\geq p^{\Diamond}$.
	\end{proof}
	
	\begin{remark}
		When $\lambda(H_1)\leq 0$, we recover $\Bu$ through $\eqref{eq:w}$ by choosing $p=0$, and obtain
		\begin{equation}\label{eq:recover u with p0}
			\Bu = \int_{\bbR} \frac{1}{\pi(1+\xi^2)}\Ct e^{i\int_0^t(\xi H_1(s)+H_2(s))\,\d s} \Bu_0\;\d \xi ,
		\end{equation}
		which is the exact formula in \cite[Theorem 1]{ALL},  and it can be seen as a special recovery from Schr\"odingerization. 
		
%		In practice, one can choose any interval $[p_1,p_2]\subset [p^{\Diamond},\infty)$ to recover $\Bu$ such that
%		\begin{equation}\label{eq:recovery by integration}
%			\Bu = \frac{1}{e^{-p_1}-e^{-p_2}}
%			\int_{p_1}^{p_2}\Bw(p)\;dp.
%		\end{equation}
	\end{remark}

	\subsection{Implementation on a quantum device}
	From~\eqref{eq:schro_dis}, $\ket{\Bw_h(t)}$ can be computed by 
	\begin{equation}\label{eq:quantum simulation dis fourier}
	\ket{\Bw_h(T)} = \big(\Phi\otimes I_n\big)\Cu(T)\big(\Phi^{-1}\otimes I_n\big)\ket{\Bw_h(0)},
	\end{equation} 
  where $\Cu(T)$ is a unitary operator, given by 
\[\Cu(T) = \Ct \exp(-i\int_0^T H^d(s)ds),\]
and $\Phi$ (or  $\Phi^{-1}$) is completed by (inverse) quantum Fourier transform (QFT or IQFT)\cite{Nielssen2000}.
	The Hamiltonian simulation with respect to $\Cu$ can be implemented as in \cite{BACS07,BCCKS,BCCKS15,BCK15,BHMT02,BN16,NB17,LC17,PQS11}.

	\subsubsection{Computation cost for the measurements}\label{sec:computation cost}
	 After the Hamiltonian simulation, one can recover the target variables for $\Bu$ by performing a measurement in the computational basis:
		\[M_k = \ket{k}\bra{k} \otimes I, \quad k \in\Ci_{\Diamond},\]
		where $\Ci_{\Diamond}  = \{j: p^{\Diamond}\leq p_j\leq p^{\Diamond}+R\}$ is referred to as the recovery index set.
		The state vector is then collapsed to
		\[ \ket{\Bu_*} \equiv \ket{k_*} \otimes \frac{1}{\mathcal{N}}\Big(\sum\limits_i w_{k_*i} \ket{i} \Big) , \quad
		\mathcal{N} = \Big(\sum\limits_i |w_{k_* i}|^2 \Big)^{1/2},\]
		for some $k_*$ in the recovery index set $\Ci_{\Diamond}$ with the probability
		\begin{align*}
			\text{P}_{\text{r}}(T,p_{k_*})
			= \frac{\sum_i |w_{k_*i}(T)|^2}{\sum_{k,i} |w_{ki}(T)|^2}
			\approx \frac{\|\Bw(T,p_{k_*})\|^2}{\sum_k \|\Bw(T,p_k)\|^2}.
		\end{align*}
		Then the likelihood of acquiring $\ket{\Bu_*}$ that satisfies $k_* \in \Ci_{\Diamond}$ is given by
		\begin{equation*}
			\text{P}_{\text{r}}^*
			=\frac{\sum_{k\in I_{\Diamond},i} |w_{ki}(T)|^2}{\sum_{k,i} |w_{ki}(T)|^2}
			\approx \frac{C_{e0}^2}{C_e^2}\frac{\|\Bu(T)\|^2}{\| \Bu(0) \|^2},
		\end{equation*}
		where
		\begin{equation}\label{Ce0}
			C_{e0} = \Big(\sum_{(p^{\Diamond}\le p_k\le (p^{\Diamond}+R)} e^{-2|p_k|}  \Big)^{1/2}, \qquad C_e = \Big(\sum_{k=0}^{N_p-1} e^{-2|p_k|}  \Big)^{1/2}.
		\end{equation}
		If $N_p$ is sufficiently large, we have
		\[\triangle p C_{e0}^2 \approx  \int_{p^\Diamond}^{p^{\Diamond+R}} e^{-2p}   \d p
		= \frac12 e^{-2p^\Diamond}(1-e^{-2R}) , \qquad
		\triangle p C_e^2 \approx  \int_{-\infty} ^{\infty}  e^{-2p} \d p= 1.\]
		Since $R\geq 1$, it yields
	$\frac{C_{e0}^2}{C_e^2} \approx \frac12 e^{-2p^\Diamond}(1-e^{-2R})>\frac{2}{5}e^{-2p^{\Diamond}}.$
	    By using the amplitude amplification, the repeated times for the measurements can be approximated as
	    \begin{equation}\label{eq:repeated times}
	    	g = \mathcal{O}\big(\frac{C_e\|\Bu(0)\|}{C_{e0}\|\Bu(T)\|}\big) = \mathcal{O}\big(e^{p^{\Diamond}}\frac{\|\Bu(0)\|}{\|\Bu(T)\|}\big).
	    \end{equation}
	    \begin{remark}
	    	If all the eigenvalue of $H_1$ are negative,  then one has $p^{\Diamond}=0$ and $C_e/C_{e0}=\mathcal{O}(1)$. 
	    	Conversely, if $H_1$ possesses positive eigenvalues, the solution exhibits exponential growth with $\|\Bu(T)\| = \mathcal{O}(e^{p^{\Diamond}}\|\Bu_0\|)$, which consequently implies $g = \mathcal{O}(1)$.
	    	For the backward heat equation discussed in \cite{JLM24SchrBackward}, the parameter $p^{\Diamond}$ is selected to truncate high-frequency Fourier modes — a choice that does not inherently require $p^{\Diamond}$ to be large; detailed justifications are provided in \cite{JLM24SchrBackward}.
	    \end{remark}

	The continuous or analog version is also similar, and it is possible in principle to implement $\Bu=e^{p^{\Diamond}}\int_{p^{\Diamond}}^{\infty} \Bw(p) dp $ directly without discretisation in the $p$ domain, using the smooth projective operator $\int_{p^{\Diamond}}^{\infty} f(p) |p\rangle \langle p| dp$, where $f(p)$ can model imperfections in the detection device. The probability of recovering the desired quantum state $\Bu(t)/\|\Bu(t)\|$ is then $(\int_{p^{\Diamond}}^{\infty} f(p) dp \|\Bu(t)\|/\Bu(0)\|)^2$ \cite{CVPDE2023}. 
	\subsection{Turning a non-autonomous system into an autonomous one}
	\label{sub:time}
	Recently, a new method was proposed  in \cite{cjL23} which can turn any  non-autonomous unitary dynamical system into an autonomous unitary system.
	First, via Schr\"odingerization, one obtains a time-dependent Hamiltonians 
	\begin{equation}
		\frac{\D}{\D t}  \Bw_h = -i H(t) \Bw_h,\quad 
		H = H^{\dagger}.
	\end{equation}
	By introducing a new ``time" variable $s$, the problem becomes a new linear PDE  defined in one higher dimension but with time-independent coefficients,
	\begin{equation}
		\frac{\partial \Bv}{\partial t} = -\frac{\partial \Bv}{\partial s} -i H(s) \Bv \qquad 
		\Bv(0,s) = \delta(s) \Bw_h(0), \quad s\in \bbR,
	\end{equation} 
	where $\delta(s)$ is a dirac $\delta$-function.  One can easily recover $ \Bw_h$ by  $\Bw_h = \int_{-\infty}^{\infty} \Bv(t,s)\;ds$. 
 
	\subsubsection{The discrete Fourier transform}
	Since $v$ decays to zero as $s$ approaches infinity, the $s$-region can be truncated to $[-\pi S, \pi S]$, where $\pi S > 4\omega +T$, with $2\omega$ representing the length of the support set of the approximated delta function. Choosing $S$ sufficiently large ensures that the compact support of the approximated delta function remains entirely within the computational domain throughout the simulation, allowing the spectral method to be applied.
	The transformation 
	and difference matrix are defined by
	\begin{equation*}
		(\Phi_s)_{lj} = (e^{i\mu^s_l(j\triangle s)}), \quad 
		D_s = \text{diag} \{\mu^s_0,\mu^s_1,\dots,\mu^s_{N_s-1}\},
		\quad \mu^s_l = (l-\frac{N_s}{2})S, \quad l,j\in[N_s],
	\end{equation*}
	where $\triangle s = 2\pi S/N_s$.
	Applying the discrete Fourier spectral discretization, it yields a time-{\it independent} system as
	\begin{equation}\label{eq:tilde v time independent}
		\frac{\D}{\D t} \tilde \Bv_h = -i \big(D_s \otimes I+I_{N}\otimes H^d\big)\tilde \Bv_h, \quad 
		\tilde \Bv_h (0) = [\Phi_s^{-1}\otimes I] (\mathbf{\delta}_{\Bh}\otimes 
		\tilde \Bw_h(0)),
	\end{equation}
	where $\mathbf{\delta}_{\Bh} = \sum\limits_{j\in [N_s]}\delta_w(s_j)\ket{j}$ with 
	$s_j = -\pi S + j\triangle s$ and 
	$\delta_{\omega}$ is an approximation to $\delta$ function defined, for example, by choosing
	\begin{equation*}
		\delta_{\omega}(x) = \frac{1}{\omega} \bigg(1-\frac{1}{2}|1+\cos(\pi \frac{x}{\omega})|\bigg)\quad |x|\leq \omega, \quad 
		\delta_{\omega}(x) = 0\quad |x|\geq \omega.
	\end{equation*}
	Here $\omega = m\triangle s$, where $m$ is the number of mesh points within the support of $\delta_{\omega}$.
	\subsubsection{The continuous Fourier transform}
	For a continuous-variable formulation, one can choose $\delta_{\omega}$ to be a Gaussian function
	$$\delta_{\omega}(x) =  \frac{1}{\sqrt{2\pi \omega^2}} e^{-x^2/(2\omega^2)}.$$
	The Fourier transform of $\delta_{\omega}$ is given by 
	$$\hat{\delta}_{\omega}(\hat{x}) = \frac{1}{2\pi}\int_{\bbR} e^{i\hat x x} \delta_{\omega}(x) dx = \frac{1}{2\pi} e^{-\frac{(\omega \hat x)^2}{2}}.$$
      Since $\hat{\delta}_{\omega}$ decays exponentially in phase space, we only need to truncate the $\hat x$ interval to a finite domain $[-\hat X, \hat X]$, where the truncation size $\hat X$  satisfies $\hat X = \mathcal{O}(\ln \epsilon^{-1})$.
	Applying the continuous Fourier spectral discretization, one gets the time-{\it independent} Hamiltonian 
	\begin{equation}\label{eq:check v time independent}
		\frac{\D}{\D t}\check{\Bv}_h = i(D_{\hat{x}}\otimes I + I\otimes H^c)\check{\Bv}_h,\quad 
		\check{\Bv}_h(0) = \hat{\bm{\delta}}_h\otimes \check{\Bw}_h(0),
	\end{equation}
	where $D_{\hat x} =\text{diag}\{\hat x_0,\hat x_1,\cdots,\hat x_{N_{\hat x}}\}$ and $\hat{\bm{\delta}}_h = \frac{1}{2\pi}\sum_{j=0}^{N_{\hat x}} e^{-\frac{(\hat x_j \omega)^2}{2}}\ket{j}$, $\hat x_j = -\hat X+2\hat Xj/N_{\hat x}$ and $\omega = 2\hat X/N_{\hat x}$. 
	
	From \eqref{eq:tilde v time independent} and \eqref{eq:check v time independent}, it is easy to find the evolution matrices are time independent.
	
	\section{Error estimates for the discretizations of the Schr\"odingerized system}
	\label{sec:error for schr}
%	In this section, we give the error estimate  for the discretizations in \eqref{eq:schro_dis} and \eqref{eq:schro_conti_computer}, respectively, of the Schr\"odingerized system.
	In this section, we establish rigorous error bounds for the discrete formulations of the Schr\"odingerized system, corresponding to the numerical schemes \eqref{eq:schro_dis}  and \eqref{eq:schro_conti_computer}, respectively.

	\subsection{Error estimates for \eqref{eq:schro_dis} }
	 
	Define the complex $(N_p+1)$ dimensional space with respect to $p$ 
	\begin{equation*}
		X_{N_p}^p = \text{span}\{e^{ik(p/L)}:-N_p/2\leq k\leq N_p/2 \}. 
	\end{equation*}
	The approximation of $\Bw$ in the finite space $X_{N_p}^p$ from the numerical solution in \eqref{eq:schro_dis} is 
	\begin{equation}\label{eq:whd(t,x,p)}
		\Bw_h^d(t,p) =  \sum_{|k|\leq N_p/2} \tilde \Bw_{j_k} e^{ik(\frac{p}{L}+\pi)}, \quad \tilde \Bw_{j_k} = (\bra{j_k}\otimes I)\tilde \Bw_h(t),
	\end{equation}
	where  $j_k = -k+N_p/2$. Correspondingly, the approximation of $\Bu$ is defined by 
	\begin{equation}\label{eq:uhd}
		\Bu_h^d(t,p) = e^{p} \Bw_h^d(t,p) \quad p\in (p^{\Diamond},p^{\Diamond}+R).
	\end{equation}
    
     From \eqref{eq:uhd}, we observe that to estimate the error between $\Bu_h^d$ and $\Bu$, it suffices to bound the error between $\Bw_h^d$ and $\Bw$. 
     It is apparent that the discretization serves as an approximation for the following system with periodic boundaries:
	\begin{equation}\label{eq:periodic w}
		\begin{cases}
			\frac{\D}{\D t} \mathcal{W}= -H_1\partial_p \mathcal{W} + iH_2 \mathcal{W}, \quad  0\leq t\leq T,\\
			\mathcal{W}(t,-\pi L) = \mathcal{W}(t,\pi L),\\
			\mathcal{W}(0,p) = \mathcal{G}(p)\Bu_0,
		\end{cases}
	\end{equation} 
	where $ \mathcal{G}$ is periodic with a period of $2\pi L$, such that 
	\begin{equation*}
		\mathcal{G}(p) = e^{-|p-(j+1)\pi L|}\quad  j\pi L \leq p <(j+2)\pi L,\quad j\in \bbZ.
	\end{equation*}
	Since $\mathcal{W}\in H^1(\Omega_p)$ is periodic, the standard estimate of spectral methods shows 
	\begin{equation}\label{eq:err peri}
		\|\Bw_h^d(T,p) - \mathcal{W}(T,p)\|_{\BL^2(\Omega_p)} \lesssim \triangle p \|\mathcal{W}(T,p)\|_{\BH^1(\Omega_p)}
		\lesssim \triangle p \|\Bu(T)\|.
	\end{equation}
	In order to estimate the error between $\Bw_h^d$ and $\Bw$, it is sufficient to analyze the error between $\Bw$ and $\mathcal{W}$.

	\begin{lemma}\label{lem:SigmaHu}
		Assume $ H(t),\, H_u(t)\in \bbC^{n,n}$ are 
		Hermitian matrices, and $\Bv\in \bbC^n$ satisfies
		\begin{equation*}
			\frac{\D}{\D t}\Bv = i(H + H_u)\Bv,\quad \Bv(0)  = \Bv_0.
		\end{equation*} 
		Then, there exits a unitary matrix $U_H$ such that
		\begin{equation}\label{eq:evolution v}
			\Bv(t) = \Ct e^{i\int_0^t H(s)\;ds}U_H(t) \Bv_0.
		\end{equation}
	\end{lemma}
	\begin{proof}
		Differentiating both sides of \eqref{eq:evolution v} with respect to $t$, it is easy to observe that $U_H(t)$ satisfies
		\begin{equation}\label{eq:U_H}
			\frac{\D }{\D t} U_H(t) = i\Ct e^{-i\int_0^t H(s)} H_u \Ct e^{i\int_0^t H(s)\;ds}U_H = i V(t)U_H ,\quad  U_H(0) = I.
		\end{equation}
		The proof is finished by noting that $V(t)$ is Hermitian.
	\end{proof}

	\begin{lemma}\label{lem:eH1e_L}
		Let $\mathcal{E}(t,p) = \mathcal{W}(t,p) - \Bw(t,p)$, it follows that 
		\begin{align}
			\int_0^T  \mathcal{E}^{\dagger}(t,-\pi L)H_1 \mathcal{E}(t,-\pi L) dt &\lesssim e^{-2\pi L + 2\lambda_{\max}^+(H_1)T} \|\Bu_0\|^2, \label{eq:int eHe1}\\ 
			\int_0^T  \mathcal{E}^{\dagger}(t,\pi L)H_1 \mathcal{E}(t,\pi L) dt &\lesssim e^{-2\pi L + 2\lambda_{\max}^-(H_1)T} \|\Bu_0\|^2. \label{eq:int eHe2}
		\end{align}
	\end{lemma}
		
		\begin{proof}
			Since $H_1$ is Hermitian, there exits a unitary $U(t)$ such that
			$H_1(t) = U(t) \Sigma(t) (U(t))^{\dagger}$, where $\Sigma(t)=\text{diag}\{\lambda_1(t), \cdots, \lambda_n(t)\}$. 
			%By letting $\hat \Bw_u = U^{\dagger}\hat \Bw $, there holds
			Introducing the change of variable $\Bw_u = U^{\dagger}\Bw$, one gets the equation of 
			$\hat{\Bw}_u = \mathscr{F}(\Bw_u)$ as 
			\begin{equation*}
				\frac{\D}{\D t} \hat \Bw_u = i\xi \Sigma \hat \Bw_u +i H_{u} \hat \Bw_u,
			\end{equation*}
			where $H_{u} = U H_2 U^{\dagger}-\partial_t U U^{\dagger}/i$ is Hermitian. Here we have used the fact that  $\partial_t U U^{\dagger}$ is skew-symmetric.  
			From Lemma~\ref{lem:SigmaHu}, 
			by choosing $p = -\pi L$, one has 
			\begin{equation*}
				\begin{aligned}
					\Bw(t,-\pi L) =  \int_{\bbR} U(t) \Lambda^w(t,\xi)
					 U_H(t,\xi) U^{\dagger}(0)\Bu_0  \,d\xi ,
					\end{aligned}
			\end{equation*}
			where $\Lambda^{w}(t,\xi)$ is a diagonal matrix with diagonal entries 
			\begin{equation*}
				\Lambda^w_{jj} = \frac{1}{\pi(1+\xi^2)} e^{i \xi (\int_0^t \lambda_j(s)ds+\pi L )},
			\end{equation*} 
			and $U_H(t,\xi)$ is a unitary matrix, satisfying
			\begin{equation*}
				\frac{\D}{\D t} U_H = i \Ct e^{-i \int_0^t \xi \Sigma(s) ds}H_u  \Ct e^{i \int_0^t \xi \Sigma(s) ds} U_H,\quad U_H(0) = I.
			\end{equation*}
			Similarly, one gets 
			\begin{equation*}
				\mathcal{W}(t,-\pi L) = \int_{\bbR} U(t) \Lambda^{\mathcal{W}}(t,\xi)
				U_H(t,\xi) U^{\dagger}(0)\Bu_0  \,d\xi, 
			\end{equation*}
			with $\Lambda^{\mathcal{W}}(t,\xi)$ the diagonal matrix, 
			\begin{equation*}
				\Lambda^{\mathcal{W}}_{jj}(t,\xi) = \hat{\mathcal{G}}(\xi) e^{i\xi (\int_0^t \lambda_j(s) ds + \pi L)}.
			\end{equation*}
			Letting $\mathcal{E}(t,p) = \mathcal{W}(t,p) - \Bw(t,p)$ and noting $U_H$ is unitary, careful calculation gives 
			\begin{equation}\label{eq:eH1e_bound}
				\begin{aligned}
			    &\int_0^T\mathcal{E}^{\dagger}(t,-\pi L) H_1 \mathcal{E}(t,-\pi L)dt \\
			   =&\int_0^T \Bu_0^{\dagger} U(0)\bigg(\int_{\bbR} U_H^{\dagger}
			    (\Lambda^{\mathcal{W}}-\Lambda^w) \Sigma (\Lambda^{\mathcal{W}}-\Lambda^w) U_H \, d\xi \bigg) U^{\dagger}(0)\Bu_0 dt\\
		   \leq \,&\big\| \Lambda\big\|\big\|\Bu_0\big\|^2,
			    \end{aligned}
			\end{equation}
			where the diagonal matrix $\Lambda$ is defined by 
			\begin{equation*}
				\Lambda = \int_{0}^T \int_{\bbR} (\Lambda^{\mathcal{W}}-\Lambda^w ) \Sigma(t)  (\Lambda^{\mathcal{W}}-\Lambda^w ) d\xi dt.
			\end{equation*}
			It follows from the inverse Fourier transform that the diagonal entries of $\Lambda $ are
			\begin{equation*}
					\Lambda_{jj} = \int_0^T \big(e^{-|\int_0^t \lambda_j(s)ds+\pi L|} - \mathcal{G}(-\int_0^t \lambda_j(s)ds-\pi L)\big)^2\lambda_j(t) dt.
			\end{equation*}
		   Changing the variable of integration by letting $q = \int_0^t \lambda_j(s)ds + \pi L$, and recalling  the definition of $\mathcal{G}$, 
			it yields
			\begin{equation}\label{eq:e-G}
				\begin{aligned}
				\Lambda_{jj} = 
					\int_{\pi L}^{\pi L + \int_0^T \lambda_j(s)ds} (e^{-|q|}-\mathcal{G}(-q))^2 d q  \lesssim 
					e^{-2\pi L + 2\lambda_{\max}^{+}(H_1)T} +e^{-2\pi L}.
				\end{aligned}
			\end{equation}	
			The proof for \eqref{eq:int eHe1} is completed by inserting \eqref{eq:e-G} into \eqref{eq:eH1e_bound}. The proof for \eqref{eq:int eHe2} is similar, which is omitted here.
		 \end{proof}

	  \begin{lemma}\label{lem:mathcal E T}
	  	Assume  $\Bw(t,p)$ and $\mathcal{W}(t,p)$ are defined in \eqref{eq:up} and \eqref{eq:periodic w}, respectively, it follows that 
	  	\begin{equation*}
	  		\|\Bw(T,p)-\mathcal{W}(T,p)\|_{\BL^2(\Omega_p)} \lesssim (e^{-\pi L +\lambda_{\max}^+(H_1)T} + e^{-\pi L +\lambda_{\max}^-(H_1)T}) \|\Bu_0\|.
	  	\end{equation*}
	  	\end{lemma}
       \begin{proof}
       	By introducing $\mathcal{E}(t,p) = \mathcal{W}(t,p) -\Bw(t,p)$, one has 
       	\begin{equation}\label{eq:err E}
       		\frac{\D}{\D t} \mathcal{E} = -H_1\partial_p \mathcal{E} + iH_2 \mathcal{E},\quad 
       		\mathcal{E}(0,p) =\begin{cases}
       			\textbf{0},\quad  \quad \quad  \quad\; \text{in} \;  \Omega_p \\
       			 \mathcal{G} - e^{-|p|},\quad  \text{in}\; \bbR \backslash \Omega_p
       		\end{cases} .
       	\end{equation}
       	Testing \eqref{eq:err E} against $\mathcal{E}^{\dagger}$ and integrating by parts, it yields 
       	\begin{equation*}
       		\int_0^T \frac{\D}{\D t} \|\mathcal{E}\|_{\BL^2(\Omega_p)}^2 = -\int_0^T\mathcal{E}^{\dagger}(t,\pi L) H_1 \mathcal{E}(t,\pi L)+\int_0^T \mathcal{E}^{\dagger}(t,-\pi L) H_1 \mathcal{E}(t,-\pi L).
       	\end{equation*}
       %	From Lemma~\ref{lem:eH1e_L}, one obtains 
       Since  $H_1$ may have negative eigenvalues, one has
       	\begin{equation*}
       		\begin{aligned}
       		\|\mathcal{E}(T,p)\|_{\BL^2(\Omega_p)}^2 & \leq \bigg|\int_0^T\mathcal{E}^{\dagger}(t,\pi L) H_1 \mathcal{E}(t,\pi L)\bigg|
       		+\bigg|\int_0^T \mathcal{E}^{\dagger}(t,-\pi L) H_1 \mathcal{E}(t,-\pi L)\bigg|.
       		\end{aligned}
       	\end{equation*}
       	The proof is finished by using Lemma~\ref{lem:eH1e_L}.
       	\end{proof}

    From Theorem \ref{thm:recovery u}, we may regard $\Bu$ as a function with respect to the variable $p$. With these preparations, we are now ready to present the main result of this subsection.
	\begin{theorem}\label{thm:err of dis}
		Suppose $\pi L$ satisfies \eqref{eq:require L} with $\epsilon$ the desired accuracy. Assume $\Bu_h^d$ is defined in \eqref{eq:uhd}. 
		There holds 
		\begin{equation*}
			\|\Bu_h^d(T) - \Bu(T)\|_{L^2(\tilde\Omega_p)} \lesssim \triangle p e^{p^{\Diamond}} \|\Bu(T)\| + \epsilon \|\Bu_0\|,
		\end{equation*}
		where $\tilde \Omega_p = (p^{\Diamond}, p^{\Diamond}+R)$,
        $R\geq 1$ is the length of the recovery region with $e^R =\mathcal{O}(1)$.
		\end{theorem}
       % {\color {green} how did you get that $\log 10$ in $R$?}
		\begin{proof}
			According to the recovery rule in Theorem \ref{thm:recovery u}, it follows that 
			\begin{equation*}
				\|\Bu-\Bu_h^d\|_{\BL^2(\tilde{\Omega}_p)}=\|e^p\Bw - e^p \Bw_h^d\|_{\BL^2(\tilde{\Omega}_p)}.
			\end{equation*}
			From the triangle equality, Lemma \ref{lem:mathcal E T} and \eqref{eq:err peri}, it yields 
			\begin{align*}
				\|\Bu - \Bu_h^d\|_{\BL^2(\tilde{\Omega}_p)} &\leq 
				\|e^p\mathcal{W} - e^p \Bw_h^d\|_{\BL^2(\tilde{\Omega}_p)} + \|e^p \Bw - e^p \mathcal{W}\|_{\BL^2(\tilde{\Omega}_p)} \notag \\
				&\lesssim \triangle p e^{p^{\Diamond}+R}\|\Bu(T)\|+ e^{-\pi L + p^{\Diamond}+R}(e^{ \lambda_{\max}^+(H_1)T}+e^{\lambda_{\max}^-(H_1)T})\|\Bu_0\|. 
			\end{align*} 
			The proof is finished by considering $e^R=\mathcal{O}(1)$ and \eqref{eq:require L}.
			\end{proof}
	
	\begin{remark}\label{second-order}
		
		It is noted from Theorem~\ref{thm:err of dis}, the limitation of the convergence order mainly comes from the non-smoothness of the initial values.
		In order to improve the whole convergence rates, we change the initial value of $\Bw(0)$ as
		$\Bw(0) = \psi(p)\Bu_0$, where $\psi(p)\in H^r(\bbR)$ satisfies $\psi(p) = e^{-p} $ for $p\geq 0$.
		Thus $\mathcal{W},\Bw\in \BH^r(\Omega_p)$ and the error estimate implies 
		\begin{align}
			\|\Bu_h^d(T) - \Bu(T)\|_{\BL^2(\tilde\Omega_p)} 
			\lesssim 
			\triangle p^r e^{p^{\Diamond}} \|\Bu(T)\|+\epsilon \|\Bu_0\|.
			\label{eq:err conti L2 one order high}
		\end{align}
The explicit construction of $\psi(p)$ is given in \cite[section 4.3]{JLM24SchrBackward}. For $r=2$, a smoother initial function is given by  
		\begin{equation} \label{eq:smooth initial}
			\psi(p) = \begin{cases}
				(-3+3e^{-1})p^3+(-5+4e^{-1})p^2-p+1 \quad p\in(-1,0),\\
				e^{-|p|} \quad \text{otherwise}.
			\end{cases}
		\end{equation}
        Assuming that $L$ is sufficiently large and the grid size   $\triangle p =\mathcal{O}(\sqrt[r]{\epsilon})$ is sufficiently small, 
        we can eliminate the constants hidden in the notation of $\lesssim$ in \eqref{eq:err conti L2 one order high}, resulting in the following formula:
        \begin{equation}\label{eq:small error high order}
        	\|\Bu(T) - \Bu_*\| \leq \frac{\epsilon}{4}\|\Bu(T)\|,
        \end{equation}
        where $\Bu_* = e^{p_{k_*}}(\bra{k_*}\otimes I) \Bw_h$, $k_*\in \Ci_{\Diamond}$.
		The extra qubits needed due to the extra dimension are 
        $\mathcal{O}\big(\log (\frac{L}{\sqrt[r]{\epsilon}}))$,
		Since $r$ can be arbitrarily large, the number of qubits required for the smooth extension in p-dimensional space achieves a nearly exponential reduction compared to the original Schr\"odingerization.
		
	\end{remark}

	\subsection{Error estimates for \eqref{eq:schro_conti_computer}}
	The approximation for $\Bw$ is computed by 
	\begin{equation}\label{eq:whc}
		\Bw_h^c  = \triangle \xi \sum_{j=1}^{N-1}\check{\Bw}_{j} e^{-i\xi_j p}
		+\frac{\triangle \xi}{2} \big(\check \Bw_N e^{-i \xi_N p}+\check \Bw_0e^{-i\xi_0 p}\big).
	\end{equation}
	Here $\check \Bw_{j} = (\bra{j}\otimes I_n)\check \Bw_h$. 
	Without considering the error of quantum simulation, one gets $\check \Bw_j = \hat \Bw(t,\xi_j)$, 
	$\xi_j = -X +2jX/N$. It is obvious to see that $\Bw_h^c$ is the numerical integration of  $\int_{-X}^{X} e^{-i\xi p} \hat\Bw\;d\xi \approx \int_{\bbR}e^{-i \xi p} \hat \Bw\;d\xi$. 
	Then, we define the approximation to $\Bu$ by 
	\begin{equation}\label{def:uhc}
		\Bu_h^c(t,p) = e^p \Bw_h^c(t,p)\quad p\in (p^{\Diamond},p^{\Diamond}+R).
	\end{equation}

	\begin{lemma}\label{thm:err of conti}
		The error estimate for \eqref{eq:up} computed by \eqref{eq:schro_conti_computer} is  given by 
		\begin{equation}\label{eq:err of whc-w}
			\|\Bw_h^c(T,p) -\Bw(T,p)\| \lesssim  (X^{-1}+ X \triangle \xi^2\|H_p^2\| )\|\Bu_0\| \quad p>p^{\diamond},
		\end{equation}
		where 
		$H_p = \int_0^T H_1(s)\;ds-pI$.
	\end{lemma}
	\begin{proof}
		The error estimate for \eqref{eq:up} by \eqref{eq:schro_conti_computer} consists of two parts. 
		The first part is from the truncation of $\xi$ domain. 
		From the definition of $\Cu_{t,0}(\xi)$ in \eqref{eq:def time operator}, it yields
		\begin{equation*}
			\|\Cu_{t,0}(\xi)e^{-i\xi p}\|\leq 1.
		\end{equation*}
		The truncation error is then bounded by
		\begin{equation}\label{eq:1st}
			\big\|\int_{(-\infty,-X]\cup [X,\infty)}\frac{\Cu_{t,0}(\xi)}{\pi(1+\xi^2)}e^{-i \xi p}\Bu_0\;d\xi \big\|
			\leq \int_{X}^{\infty} \frac{2 \|\Bu_0\|}{\pi(1+\xi^2)}\;d\xi \lesssim  X^{-1}\big\|\Bu_0\big\|.
		\end{equation}
		The second part is from the numerical integration. Write $F(t,\xi,p) = \frac{\Cu_{t,0}(\xi)}{\pi(1+\xi^2)}e^{-i\xi p}$ and $H_p = \int_0^t H_1(s)\;ds-pI$, one then has
		\begin{equation*}
			\partial_{\xi}^2 F(t,\xi,p)= 
			\bigg(-\frac{H_p^2}{\pi(1+\xi^2)}
			-\frac{4 i\xi H_p+2}{\pi (1+\xi^2)^2}
			+\frac{8\xi^2 }{\pi (1+\xi^2)^3}\bigg)\Ct e^{i \int_0^t \xi (H_1-\frac{p}{t}I)+ H_2(s)\;ds}.
		\end{equation*}
		According to the error estimate of numerical quad,
		there exits $\xi^*\in [-X,X]$ such that 
		\begin{equation}\label{eq:2nd}
			\big\|\Bw_h^c-\int_{-X}^X 
			\frac{e^{-i\xi p}\Cu_{t,0}(\xi) \Bu_0}{\pi(1+\xi^2)}d\xi\big\|
			=\frac{\triangle \xi^2 X}{6}
			\|\partial_{\xi}^2 F(t,\xi^*,p)\|\|\Bu_0\|
			\lesssim \triangle \xi^2 X \|H_p^2\|\|\Bu_0\|.
		\end{equation}
		The proof is completed by the triangle inequality from \eqref{eq:1st} and \eqref{eq:2nd}.
	\end{proof}

		Integrating the inequality of \eqref{eq:err of whc-w}, the error between $\Bu_h^c$ and $\Bu$ is shown as follows.
		
		\begin{theorem}
			Assume $X$ satisfies \eqref{eq:require X}, $\Bu_h^c$ is defined in \eqref{def:uhc}. It follows that 
			\begin{equation}\label{eq:uhc-u}
				\|\Bu_h^c - \Bu\|_{\BL^2(\tilde{\Omega}_p)} \lesssim \bigg(\epsilon + \epsilon^{-1}\triangle \xi^2  \big((p^{\Diamond})^2+\big\|\int_0^TH_1(s)ds\big \|^2\big)\bigg) e^{p^{\Diamond}}\|\Bu_0\|,
			\end{equation}
				where $\tilde \Omega_p = (p^{\Diamond}, p^{\Diamond}+R)$, and  $R\geq 1$ is the length of the recovery domain with $e^R = \mathcal{O}(1)$.
			\end{theorem}
	\begin{remark}
		It is noted that the first part $\epsilon$ dominates in \eqref{eq:uhc-u} when
		$\triangle \xi$ is small enough such that $\triangle \xi \lesssim  \epsilon/\sqrt{(p^{\Diamond})^2+\big\|\int_0^TH_1(s)ds\big \|^2}$.
        By using the Fourier transform of smoother initial data,  we no longer need to truncate the interval at $X=\mathcal{O}(1/\epsilon)$,and can instead use the much smaller
		cutoff $X = \mathcal{O}(\log(1/\epsilon)^{1/\beta})$, $\beta \in (0,1)$ .  This leads to an exponential reduction in the Hamiltonian simulation time caused by the extra dimension 
		with respect to $\epsilon$ when compared to the original  formula \eqref{eq:schro_conti_computer} \cite{ACL23}.
	\end{remark}
	\subsection{Complexity analysis}
	In order to get the total query complexity, we need to obtain the cost of 
    simulating the time-evolution operator $\Cu(T)$.
	Since a time-dependent system can be made autonomous (see section \ref{sub:time}), we focus on the time-independent cost of $\Cu(T)$ for the discrete Fourier transform of Schr\"odingerization. The complexity of the continuous Fourier transform is similar to \cite{ACL23} and is omitted here.
	 We introduce the  
	the complexity of Hamiltonian simulation in \cite{gilyen2019quantum} by QSVT.
	\begin{lemma}[Complexity of block-Hamiltonian simulation]\label{lem:hamiltonian complexity}
		Let \(\epsilon \in (0, \frac{1}{2})\), \(t \in \mathbb{R}\), and \(\alpha \in \mathbb{R}^+\). Let \(U\) be an \((\alpha, a, 0)\)-block-encoding of the unknown Hamiltonian \(H\). To implement an \(\epsilon\)-precise Hamiltonian simulation unitary \(V\) that is a \((1, a + 2, \epsilon)\)-block-encoding of \(e^{\mathrm{i}tH}\), the unitary \(U\) must be queried
		\begin{equation}\label{eq:ham-query}
			\Theta\left( 
			\frac{\alpha|t| + \log(1/\epsilon)}{\log\left(e + \log(1/\epsilon)/(\alpha|t|)\right)} 
			\right)
		\end{equation}
		times.
	\end{lemma}
		
		Finally, we obtain the complexity of our algorithm based on Lemma \ref{lem:hamiltonian complexity}, error estimates and repeated times discussed in section \ref{sec:computation cost}. 
		\begin{theorem}\label{thm:complexity}
			Assume that $\psi \in H^r(\bbR)$ is smooth enough, $\pi L$ is large enough and 
			$\triangle p = \mathcal{O}(\sqrt[r]{\epsilon})$ is small enough 
			to satisfy \eqref{eq:small error high order}
			with  $\epsilon$ the desired precision. 
			There exits a quantum algorithm that prepares an $\epsilon$-approximation of the state $\ket{\Bu}$ with 
			$\Omega(1)$ success probability and a flag indicating success, using 
			\begin{equation*}
				\tilde{\mathcal{O}}\bigg(e^{p^{\Diamond}} \frac{\|\Bu(0)\|}{\|\Bu(T)\|}\big(\frac{T \alpha_H}{\sqrt[r]{\epsilon}}+ \log  \frac{e^{p^{\Diamond}}\|\Bu(0)\|}{\epsilon^{1/r+1}\|\Bu(T)\|}\big)\bigg)
			\end{equation*}
			queries, where $\alpha_H \geq \|H_i\|,i=1,2$.
			\end{theorem}
			\begin{proof}
				Based on the Hamiltonian simulation, the state $\Bw^a/\eta_0$ can be prepared with $\eta_0 =\|\Bw_h(0)\|$, yielding the approximation state vector $\ket{\Bu^a(T)}$
				for the solution $\ket{\Bu(T)}$. The error between $\ket{\Bu}$ and $\ket{\Bu^a(T)}$ is obtained by 
				\begin{equation*}
					\|\ket{\Bu(T)}-\ket{\Bu^a(T)}\| \leq \|\ket{\Bu(T)} - \ket{\Bu_*}\| + \|\ket{\Bu_*} -\ket{\Bu^a(T)}\|,
				\end{equation*} 
				where $\Bu_* = e^{p_{k_*}}(\bra{k_*}\otimes I)\Bw_h$, $k_*\in \Ci_{\Diamond}$.
				From \eqref{eq:small error high order}, one has
				\[\|\ket{\Bu(T)} - \ket{\Bu_*}\| \leq 2 \|\Bu(T)-\Bu_*\|/\|\Bu(T)\| \leq \epsilon/2.\]
				Neglecting the error in the implementation of $\Cu(T)$, there holds 
				\begin{equation}
					\begin{aligned}
					\|\ket{\Bu_*}-\ket{\Bu^a(T)}\| \leq  2 \|\Bu_* - \Bu^a\|/\|\Bu_*\|
					\leq 2e^{p_k} \|\Bw_h - \Bw^a\|/\|\Bu_*\|.
					\end{aligned}
				\end{equation}
				To ensure the overall error remains within $\epsilon$, we require $\|\ket{\Bu_*} - \ket{\Bu^a}\|\leq \epsilon/2$, yielding
				\begin{equation}\label{eq:wh-wha}
					\|\ket{\Bw_h} - \ket{\Bw^a}\| \leq 
					2\frac{\|\Bw_h - \Bw^a\|}{\|\Bw_h\|} \leq 
					\frac{\epsilon e^{-{p_{k_*}}}\|\Bu_*\|}{2\eta_0}\approx \frac{\epsilon e^{-p_{k_*}} \|\Bu(T)\|}{\triangle p \|\Bu(0)\|}:=\delta.
				\end{equation}
				According to Lemma \ref{lem:hamiltonian complexity}, the assumption of mesh size and \eqref{eq:wh-wha},
				the time cost to implement $\Cu(T)$ is 
				$$ \tilde{\mathcal{O}}\left(T \|H^d\|+ \log \frac{1}{\delta} \right)
				=\tilde{\mathcal{O}}\left( \frac{T\alpha_H}{\sqrt[r]{\epsilon}} 
				+ \log \frac{e^{p^{\Diamond}} \|\Bu(0)\| }{ \epsilon^{1/r+1} \|\Bu(T)\|}\right).
				$$
				The proof is finished by multiplying the repeated times shown in \eqref{eq:repeated times}. 
				\end{proof}
			
		    \begin{remark}
		    	Since $r$ can be arbitrarily large and we use spectral methods to discretize in the $p$-direction,
                the dependence on $1/\epsilon$ can be reduced to a logarithmic level, i.e., 
                the query complexity approaches $\tilde{\mathcal{O}}\big(e^{p^{\Diamond}}\frac{\|\Bu(0)\|}{\|\Bu(T)\|}T \alpha_H\log \frac{1}{\epsilon} \big)$.
                Following the analysis in \cite{ACL23}, we derive the complexity of the continuous Fourier transform as
                $\tilde{\mathcal{O}}\big(e^{p^{\Diamond}} \frac{\|\Bu(0)\|}{\|\Bu(T)\|} T\alpha_H (\log\frac{1}{\epsilon})^{1/\beta}\big)$,
                where $\beta$ is a real parameter satisfying $0<\beta <1$.
		    	\end{remark}

	\section{The inhomogenous problem}
	\label{sec:the inhomogenous problem}
	In this section, we consider the problem with inhomogeneous source term, i.e. $\Bb(t)\neq \bm{0} $ in \eqref{eq:ODE}.
	Define $\tilde A\in \bbC^{2n\times 2n}$ by $\tilde A = \begin{bmatrix}
		A &B \\
		O &O
	\end{bmatrix}$ and 
	$H_1^{A} = (A+A^{\dagger})/2$, $H_2^A = (A-A^{\dagger})/(2i)$.
	There holds $\tilde A = \tilde H_1 +i \tilde H_2$, with  $\tilde H_1$ and $\tilde H_2$
	both Hermitian, where $\tilde H_1$ and $\tilde H_2$ are defined by 
	\begin{equation}\label{eq:large matrix H1}
		\tilde H_1 =  \begin{bmatrix}
			H_1^A   &\frac{B}{2}\\
			\frac{B^{\top}}{2}   & O
		\end{bmatrix},\quad 
		\tilde H_2 = \begin{bmatrix}
			H_2^A  &\frac{B}{2i}\\
			-\frac{B^{\top}}{2i} &O
		\end{bmatrix}.
	\end{equation}
	Apply the Schr\"odingerization to \eqref{eq: homo Au}, it leads  to 
	\begin{equation}\label{eq:w source term}
		\frac{\D}{\D t} \Bw_f 
		= - \tilde{H}_1\partial_p \Bw_f + i \tilde{H}_2 \Bw_f,\quad 
		\Bw_f(0) = \psi(p) \begin{bmatrix}
						\Bu_0\\
						\Br_0
					\end{bmatrix}.
	\end{equation}

	\subsection{The eigenvalues of $\tilde H_1$ impacted by the inhomogeneous term}
	Assume 
	$H_1^A$  and $\tilde H_1$ have $n$  and $2n$ real eigenvalues, respectively,  such that
	\begin{equation}\label{eq:eigenvalue of H1A and tilde H1}
		\lambda_1( H_1^A)\leq \lambda_2(H_1^A)\leq \cdots \leq 
		\lambda_{n}(H_1^A),\quad
		\lambda_1(\tilde H_1)\leq \lambda_2(\tilde H_1)\leq \cdots \leq \lambda_{2n}(\tilde H_1),
	\end{equation}
	for all $t\in [0,T]$.
	
	\begin{lemma}\label{lem:eigenvalues of H1}
		Assume  $\Bb \neq \textbf{0}$, then there exit both positive and negative eigenvalues of $\tilde H_1$.
	\end{lemma}
	\begin{proof}
		We first consider the case that $H_1^A$ is invertible. After elementary row and column operations, one has 
		\begin{equation}\label{eq:eigenvalue of tilde HRC}
			\begin{bmatrix}
				I   &O\\
				-\frac{B^{\top}(H_1^A)^{-1}}{2}&I
			\end{bmatrix}
			\tilde H_1
			\begin{bmatrix}
				I   &\frac{(H_1^A)^{-\top}B}{2}\\
				O   &I
			\end{bmatrix}=
			\begin{bmatrix}
				H_1^A & O\\
			O & -\frac{B^{\top} (H_1^A)^{-1}B}{4}
			\end{bmatrix} = \tilde H_1^{RC}.
		\end{equation}
		According to \eqref{eq:eigenvalue of tilde HRC}, it follows that $\lambda(\tilde H_1^{RC}) = \{\lambda(H_1^A),-\lambda(B^{\top} (H_1^A)^{-1} B)/4\}$ have both positive and negative values.
		%Since the number of positive  negative) eigenvalues of matrix $\tilde H_1$ is the same as that of matrix $\tilde H_1^{RC}$.
		Since the matrix $\tilde H_1$ has the same number of positive eigenvalues and the same number of negative eigenvalues as the matrix $\tilde H_1^{RC}$, the statement is proved.
		
		Next we consider the case that $ H_1^A$ is not invertible.
		The following proof is from eigenvalue inequalities for Hermitian matrices \cite[Corollary 4.3.15]{Horn12}, that is  for any  Hermitian matrices $H_A, H_B \in \bbC^{n,n}$, then 
		\begin{equation}\label{eq:eigenvalues of Hermitian}
			\lambda_i(H_A)+\lambda_1(H_B)\leq \lambda_i(H_A+H_B)\leq \lambda_i(H_A)+\lambda_n (H_B),\quad 
			i=1,\cdots,n.
		\end{equation}
		Let $H^*_N = H_1^A - \delta_0 I$, with $\delta_0 = 2 \max\{|\lambda(H_1^A)|\}$. 
		Then $H^*_N$ is negative. By letting $\tilde H^*_N = \begin{bmatrix}
			H^*_N &B/2\\
			B^{\top}/2 &O
		\end{bmatrix}$, it follows immediately from what we have proved that 
		there exits $i_0$ such that $\lambda_{i_0}(\tilde H^*_N)>0$. From \eqref{eq:eigenvalues of Hermitian}, it is easy to see that 
		$\lambda_{i_0}(\tilde H_1) \geq \lambda_{i_0}(\tilde H^*_N)>0$. Similarly, there exits $j_0$ such that 
		$\lambda_{j_0}(\tilde H_1) \leq \lambda_{j_0}(\tilde H^*_P)<0$, where $\tilde H^*_P = \begin{bmatrix}
			H^*_P &B/2\\
			B^{\top}/2 &O
		\end{bmatrix} $ with $H^*_P = H_1^A +\delta_0 I$.
	\end{proof}
	\begin{lemma}\label{lem:bound of lambda}
		%There exists a set of indices $\{i_k\}_{k=1}^n$ such that 	
		The eigenvalues of $\tilde H_1$ satisfy
		\begin{equation}
			|\lambda_{i_k}(\tilde{H}_1)- \lambda_k(H_1^A) |\leq \frac{|b|}{2} 
			\quad \text{and}\quad 
			|\lambda_{j_k}(\tilde H_1)|\leq \frac{|b|}{2}, \quad 1\leq i_k\leq 2n, \quad 1\leq k\leq n,
		\end{equation}
		where $|b| = \max\limits_{t}\|\Bb\|_{l^{\infty}}$, and the indices satisfy  $j_k \in \{1,2,\cdots,2n\}\backslash (\{i_k\}_{k=1}^n)$.
	\end{lemma}
	\begin{proof}
		The result will be proved from \eqref{eq:eigenvalues of Hermitian}. 
		We complete the proof by observing that
		\begin{equation}
			\tilde H_1 = \tilde H_{11} +\tilde H_{12}=\begin{bmatrix}
				H_1^A &O\\
				O & O
			\end{bmatrix}
			+\begin{bmatrix}
				O & \frac{B}{2}\\
				\frac{B^{\top}}{2} &O
			\end{bmatrix},
		\end{equation}
		with $|b|/2$ and $-|b|/2$ the maximum and minimum eigenvalues of $\tilde H_{12}$.
	\end{proof}
	
	\subsection{Turning an inhomogeneous system into a homogeneous one}

	From Lemma~\ref{lem:eigenvalues of H1}, it can be seen that the presence  of the source term will destroy the negative definiteness of matrix $H_1^A$.
	Then we alleviate  this defect by rescaling the auxiliary vector $\Br$ with $\Br /\gamma$, where 
	\begin{equation}\label{eq:gamma}
		\gamma=1/\big(C_T|b|\big), \quad 
		|b| = \max\limits_{t}\|\Bb\|_{l^{\infty}}.
	\end{equation}
	Here, $C_T$ is a parameter that satisfies $C_T = \mathcal{O}(1)$ and is larger than the evolution time. If the evolution time $T$ is known, we can set $C_T = T$.
	The new linear system is 
	\begin{equation}\label{eq:varepsilon A}
		\frac{\D}{\D t} \Bu_f^{\gamma} 
		= A^{\gamma} \Bu_f, \quad 
		A_f  =  \begin{bmatrix}
			A & \gamma B\\
			O &O
		\end{bmatrix},\quad 
	\Bu_f(0) = \begin{bmatrix}
		\Bu_0 \\
		\Br_0 /\gamma
	\end{bmatrix}. 
	\end{equation} 
	Define 
	\begin{equation*}
		H_1^{\gamma} = \begin{bmatrix}
			H_1^A & \frac{\gamma B}{2}\\
			\frac{\gamma B^{\top}}{2} &O
		\end{bmatrix},\quad 
		H_2^{\gamma}=
		\begin{bmatrix}
			H_2^A & \frac{\gamma B}{2i}\\
			-\frac{\gamma B^{\top}}{2i} &O
		\end{bmatrix}.
	\end{equation*}
	Using the Schr\"odingerization method, one gets
	\begin{equation}\label{eq:up source term gamma}
		\frac{\D}{\D t} \Bw_f^{\gamma} = 
		-H_1^{\gamma} \partial_p \Bw_f^{\gamma} 
		+i H_2^{\gamma} \Bw_f^{\gamma},\quad 
		\Bw_f^{\gamma}(0) =e^{-|p|}\begin{bmatrix}
		\Bu_0\\
		\Br_0/\gamma
		\end{bmatrix}.
	\end{equation}
	From Theorem~\ref{thm:recovery u}, one could choose 
	$p>p^{\gamma,\Diamond}$ to recover $\Bu_f^{\gamma}$ from $\Bw_f^{\gamma}$, where 
	\begin{equation}\label{eq:choice region}
		p^{\gamma,\Diamond} =p^{\diamond}+\mathcal{O}(\gamma |b|T)=\lambda_{\max}^+(H_1^A)T+ \mathcal{O}(1).
	\end{equation}
	Then, we arrive at the recovery theorem for the inhomogeneous case.
	\begin{theorem}
				Assume the eigenvalues of $H_1^A$ satisfy \eqref{eq:eigenvalue of H1A and tilde H1} and $\gamma$ satisfies \eqref{eq:gamma}, we have
				\begin{equation*}
					%	\Bw^{\gamma}(t,p) = \Bw(t,p) \quad \text{for}\; p>p^{\gamma,\diamond},
				    \Bu = e^p(\bra{0}\otimes I)\Bw_f^{\gamma},\quad p>p^{\gamma,\Diamond}, 
					\end{equation*}
				where $p^{\gamma,\diamond}$ is given in \eqref{eq:choice region}, or use the integration to get
				\begin{equation*}
					\Bu = e^p (\bra{0}\otimes I) \int_{p}^{\infty}\Bw_f^{\gamma}(q)dq,\quad 
					p>p^{\gamma,\Diamond}.
				\end{equation*}
		\end{theorem}

	\subsubsection{Discretization of the Schr\"odingerized system}
	Applying the discrete Fourier transform to \eqref{eq:up source term gamma} gives
	\begin{equation}\label{eq:shro dis source}
		\frac{\D}{\D t}
		\tilde \Bw_{f,h}^{\gamma}
		= -i \big(D_p \otimes H_1^{\gamma} - I\otimes H_2^{\gamma} \big)
		\tilde{\Bw}_{f,h}^{\gamma},\quad 		
		\tilde{\Bw}_{f,h}^{\gamma} = 
		(\Phi^{-1}\otimes I)(\psi_h \otimes \begin{bmatrix}
			\Bu_0 \\
			C_T|b|\Br_0
		\end{bmatrix}),
	\end{equation}
	where  $\psi_h = \sum_{k\in [N_p]} e^{-|p_k|}\ket{k}$ and $p_k = -\pi L + 2k\pi L/N_p$. Here $\pi L$ is large enough to  satisfy 
	\begin{equation}\label{eq:pi L inhomo}
		e^{-\pi L + 2\lambda_{\max}^+(H_1^A)T+R+1}\leq \epsilon,\quad
		e^{-\pi L +\lambda_{\max}^-(H_1^A)T+\lambda_{\max}^+(H_1^A)+R+1}\leq \epsilon,
	\end{equation}
	with $R\geq 1$ and $e^R=\mathcal{O}(1)$.
	Then, the Hamiltonian simulation can be performed.

    According to Theorem \ref{thm:err of dis}, we get the error estimate of system in  \eqref{eq:shro dis source}.
	\begin{corollary}
		Assume $\pi L$ is large enough and satisfies \eqref{eq:pi L inhomo}. The approximation $\Bu_h^d$ is defined by 
		\[\Bu_h^d(T,p) = e^p \sum_{|k|\leq N_p/2}(\bra{j_k}\otimes I) \tilde{\Bw}_{f,h}^{\gamma} e^{ik(p/L+\pi)}\quad p\in \tilde{\Omega}_p,\]
		where $j_k = -k+\frac{N_p}{2}$ and $\tilde{\Omega}_p = (p^{\gamma,\Diamond},p^{\gamma,\Diamond}+R)$. Then, it yields 
		\[ \|\Bu_h^d - \Bu\|_{L^2(\tilde\Omega_p)}\lesssim \triangle p  e^{p^{\Diamond}} \big(\|\Bu(T)\| +C_T\||\Bb|\|\big)  + \epsilon \big(\|\Bu_0\| + C_T\||\Bb|\|\big),\]
		where $\||\Bb|\| = |b|\|\Br_0\|.$
	\end{corollary}

	Since $\Bu_f^{\gamma} = \begin{bmatrix}
		\Bu \\ \Br/\gamma
	\end{bmatrix}$, one can perform a projection to get $\ket{\Bu(T)}$ with the probability
	$\frac{\|\Bu(T)\|^2}{\|\Bu_f^{\gamma}(T)\|^2}$. According to the analysis of the computation cost for the measurement in section \ref{sec:computation cost}, the overall probability for getting $\ket{\Bu(T)}$ is approximated by 
	\begin{equation*}
		P_u = \frac{C_{e0,\gamma}^2}{C_{e,\gamma}^2}\frac{\|\Bu(T)\|^2}{\|\Bu_f^{\gamma}(0)\|^2}
		= \frac{C_{e0,\gamma}^2}{C_{e,\gamma}^2}\frac{\|\Bu(T)\|^2}{\|\Bu(0)\|^2 + C_T^2|\|\Bb\||^2},
	\end{equation*}
	where $|\|\Bb\||=|b|\|\Br_0\|$, and $\frac{C_{e0,\gamma}}{C_{e,\gamma}} \approx  \sqrt{\frac{(1-e^{-2})e^{-2p^{\gamma,\Diamond}}}{2}} \geq 0.65e^{-p^{\gamma,\Diamond}}$.  By using the amplitude  amplification, the repeated times for the measurements can be approximated as 
	\begin{equation}\label{eq:inhomo repeated times}
		g \approx e^{p^{\Diamond}+1} \frac{\sqrt{\|\Bu(0)\|^2 + C_T^2|\|\Bb\||^2}}{\|\Bu(T)\|}.
	\end{equation}
	As shown in Theorem \ref{thm:complexity}, for smooth initial data and with $\gamma$ selected as in \eqref{eq:gamma}, the Hamiltonian simulation of \eqref{eq:shro dis source} achieves a query complexity of nearly 
	$ \tilde{\mathcal{O}}(e^{p^{\Diamond}} \frac{\sqrt{\|\Bu(0)\|^2 + C_T^2|\|\Bb\||^2}}{\|\Bu(T)\|} (T\alpha_H \log\frac{1}{\epsilon}))$, where $\alpha_H \geq \|H_i^A\|$, $i=1,2$.  
	
	By using the Fourier transform in $p$ of $\Bw_f^{\gamma}$, it yields
	\begin{equation}\label{eq:hat w y}
		\frac{\D}{\D t} \hat \Bw_f^{\gamma} 
		=i(\xi H_1^{\gamma} + H_2^{\gamma})
		\hat \Bw_f^{\gamma} =  i H^{\gamma} \hat{\Bw}_f^{\gamma},
	\end{equation}
	where $\hat\Bw_f^{\gamma} = \mathscr{F}(\Bw_f^{\gamma})$. Here 
	$ H^{\gamma}(\xi)$ is Hermitian for any $\xi \in \bbR$. 
	From Lemma~\ref{lem:bound of lambda},  one has $\lambda(H^{\gamma}_1) = \lambda(H_1^A) +\mathcal{O}(1)$.
	
	By applying the continuous Fourier transform to the Schr\"odingerization framework (see \eqref{eq:hat w y}), we truncate the $\xi$-domain to $[-X, X]$ for numerical implementation, yielding 
	\begin{equation}\label{eq:check w source term}
		\frac{\D}{\D t} \check{\Bw}_{f,h}^{\gamma} = i (D_{\xi} \otimes H_1^{\gamma} + I\otimes H_2^{\gamma}) \check{\Bw}_{f,h}^{\gamma},\quad 
		\check{\Bw}_{f,h}^{\gamma} = \bm{\xi}_h  \otimes
				\begin{bmatrix}
						\Bu_0\\
						C_T|b|\Br_0
				\end{bmatrix},
	\end{equation}
	where $D_{\xi}$ and $\bm{\xi}_h$ are defined in the same way as in  \eqref{eq:schro_conti_computer}. Following Theorem~\ref{thm:err of conti}, one gets the  error estimate for \eqref{eq:check w source term}.
	\begin{corollary}\label{thm:err shcr conti source term}
		Assume $\Bw_h^c$ is obtained by \eqref{eq:whc} with the solution to \eqref{eq:check w source term}, and $\Bu_h^c = e^p \Bw_h^c$. 
		Additionally, $X$ is large enough to satisfy \eqref{eq:require X}.
		Then there holds 
		\begin{align*}
			\|\Bu_h^c - \Bu\|_{L^2(\tilde\Omega_p)}\lesssim \big(\epsilon+\frac{1}{\epsilon}\triangle \xi^2 ((p^{\Diamond})^2+\|\int_0^T (H_1^{A}(s)ds\|^2)\big)e^{p^{\Diamond}}(\|\Bu_0\|+C_T\||\Bb|\|),
		\end{align*}
		where $|\|\Bb\||=|b|\|\Br_0\|$ and  $\tilde{\Omega}_p = (p^{\gamma,\Diamond},p^{\gamma,\Diamond}+R)$.
	\end{corollary}
	
	 \begin{remark}
	    Our algorithm by using \eqref{eq:shro dis source} focuses on direct quantum state preparation
	    of
	    $$ \Bu_* = e^{p_{k_*}}(\bra{k_*}\otimes \bra{0}\otimes I) \Bw_{f,h}^{\gamma},$$
	    with $\Bw_{f,h}^{\gamma} = (\Phi^{-1} \otimes I)\tilde{\Bw}_{f,h}^{\gamma}$, enabling the computation of a wide range of
	    physical quantities.
	    To estimate the quantum observable \(\bra{\bm{u}_*} S \ket{\bm{u}_*}\) associated with the quadratic form \(\bm{u}^\dagger S \bm{u}\), we proceed as follows. Given access to a block-encoding \(U_S\) of the operator \(S\), both the real and imaginary components of \(\bm{u}_*^\dagger S \bm{u}_*\) can be efficiently computed. Specifically, these quantities are obtained via the Hadamard test technique \cite{TAWL21}, where the real component corresponds to the standard protocol and the imaginary part is extracted by introducing a \(\pi/2\)-phase shift in the ancilla qubit rotation.
	    To get the quantum observable after computing \eqref{eq:check w source term},
	    the process can be found in \cite{ALL}.
	\end{remark}
	\section{Numerical tests}
	\label{sec:numerical tests}
	
	In this section, we use several examples to verify our theory and show the correctness of the mathematical formulation.  All of the numerical tests
	are performed in the classical computers by using Crank-Nicolson method for temporal discretization.
%	with the tiny time step $\triangle t = \frac{1}{2^{10}}$.

   	\subsection{Recovery from Schr\"odingerization}
   
   In this test, we use the following  scattering model-like problem in $[0,2]$ to test the recovery from Schr\"odingerization,
   \begin{equation}
   	\partial_t u = \Delta u +k^2u,\quad u(0) =\sin(\pi x),
   \end{equation}
   with zero boundary conditions.  The exact solution is   $u = e^{(k^2-\pi^2)t} \sin(\pi x)$, and we use $k=4$. The  spatial discretization is given by the  finite difference method, 
   \begin{equation*}%\label{eq:spatial discretization scattering}
   	\frac{\D}{\D t}\Bu = A\Bu, \quad 
   	A = \begin{bmatrix}
   		\alpha &\beta & & & &\\
   		\beta &\alpha & \beta & & &\\
   		%     &\beta  & \alpha & \beta & \\
   		&\ddots  &\ddots  &\ddots  &\\
   		& &\beta  & \alpha & \beta\\
   		& & & \beta & \alpha
   	\end{bmatrix},\quad 
   	\alpha = -\frac{2}{h^2}+k^2,\;
   	\beta = \frac{1}{h^2}.
   \end{equation*}
    The computation stops at $T=1$.
   One can check that the eigenvalues of $A$ are $\lambda_j(A) = \frac{1}{h^2}(-2+2\cos(\frac{j\pi}{n+1})) +k^2$, $j=1,\cdots,n$.
   It can be seen that the eigenvalue of $A$ may be positive  when  $k$ is large.
   In this test, we observe $\lambda_{\max}^+(A)\approx 6$ when $h=1/2^5$ numerically.
   From  the plot on the left in Fig.~\ref{fig:recovery of u}, it can be seen that the error between $\Bu$ and $\Bw_h^c e^{p}$ drops precipitously at $p^{\Diamond}=\lambda_{\max}^+(H_1)T\approx 6$.
   Here $\Bw_h^c$ is computed using \eqref{eq:schro_conti_computer}, with $\xi$ truncated to the interval $(-80,80)$, $\triangle \xi = \frac{5}{2^4}$ and $\triangle t = \frac{1}{2^5}$.
   According to Theorem~\ref{thm:recovery u}, we should choose $p>p^{\Diamond}\approx 6$ to recover $\Bu$.  The results are shown on the right of Fig.~\ref{fig:recovery of u},
   where the numerical solutions  recovered by selecting a single point or numerical integration are close to the exact solutions.

   \begin{figure}[t!]
   	\includegraphics[width=0.45\linewidth]{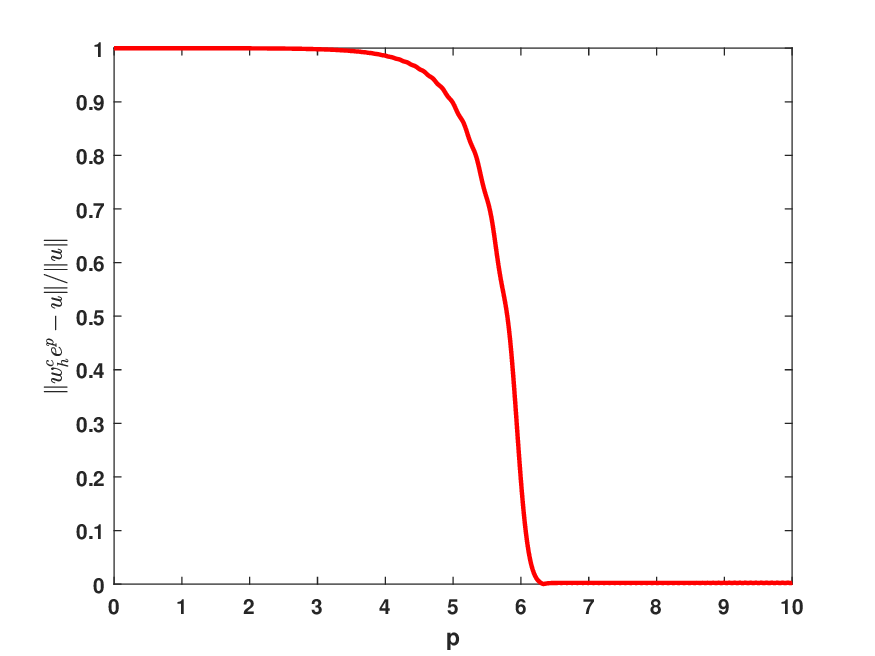}
   	\includegraphics[width=0.45\linewidth]{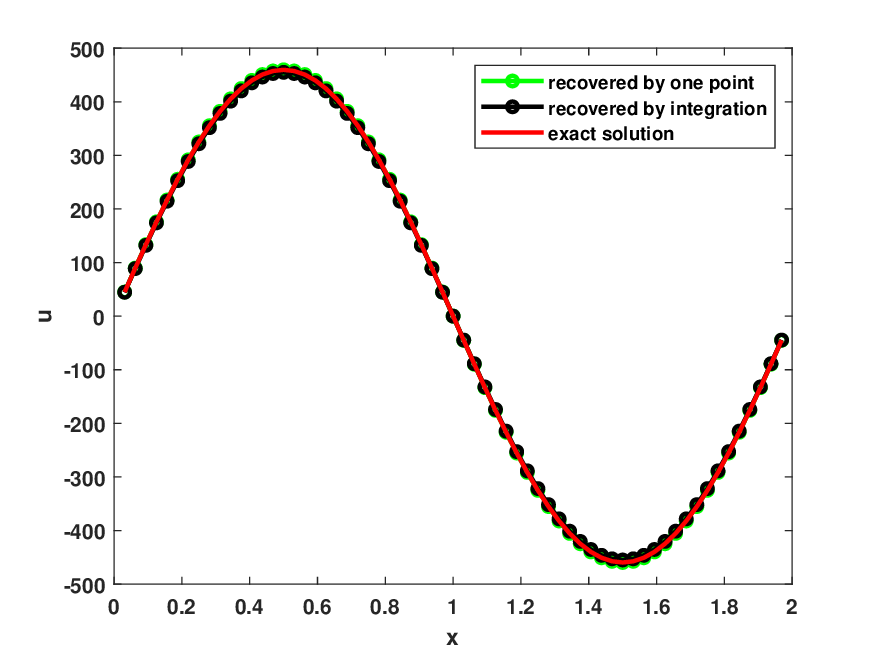}
   	\caption{ Left: error of $\|\Bu - \Bw_h^c(p)e^{p}\|/\|\Bu\|$
   		with respect to $p$ with $\Bw_h^c$ computed by \eqref{eq:whc}.  Right: the recovery from Schr\"odingerization by choosing $p>p^{\diamond} = \lambda_n(H_1)  T\approx 6$.	
   	}\label{fig:recovery of u}
   \end{figure}
	
	\subsection{Recovery from Schr\"odingerization for ill-posed problems}
	In this test, we use the backward heat equation in $[0,2]$, with the Dirichlet boundary condition $u(t,0)=u(t,2)=0$, to test the recovery from Schr\"odingerization,
	\begin{equation*}
		\partial_t u = -\Delta u, \quad u(0,x) = \exp(-25\pi^2) \sin(\frac{\pi}{2}x).
	\end{equation*}
	The exact solution is given by 
	%\begin{equation*}
	$	u(t,x) = \exp(\frac{\pi^2}{4}t-25\pi^2)\sin(\frac{\pi}{2}x) $.
	%\end{equation*}
    It is well known that the backward heat equation is unstable, which is hard to simulate. 
	Applying the Schr\"odingerization yields 
	\begin{equation}\label{w-heat}
		\frac{\D}{\D t} w =  \Delta \partial_p w,\quad w(0,x,p) = e^{-|p|}u(0,x).  
	\end{equation}
	By using the Fourier transform technique to \eqref{w-heat} with respect to the variables $p$ and $x$, which correspond to Fourier modes $\xi$ and $\eta$, respectively, one gets
	the Fourier transform $\hat w(t,\eta,\xi)$ of the exact solution $w(t,x,p)$ of problem~\eqref{w-heat} satisfying
	\begin{equation*}%\label{eq:Fourier transform w}
		\frac{\D}{\D t} \hat w  = -i\xi\eta^2 \hat{w},\quad 
		\hat w(0) = \hat w_0  = \frac{\hat u_0(\eta)}{\pi (1+\xi^2)},
	\end{equation*}
	where $\hat{u}_0(\eta) = \mathscr{F}(u_0) = \frac{1}{2\pi} \int_{\bbR} e^{i x \eta} u(0,x)dx$.
	The solution of $w$ is then found to be 
	\begin{equation*}
		\begin{aligned}
	   w(T,x,p)= \frac{1}{2i}\int_{\bbR} e^{-|p-\eta^2 T|} e^{-ix\eta-25\pi^2 }
	 \big(\delta(\eta+\frac{\pi}{2})- \delta(\eta-\frac{\pi}{2}) \big) d\eta.
	 \end{aligned}
	\end{equation*}
	Since $\hat{u}_0$ has a compact support, namely $\hat{u}_0 = 0$ for $|\eta|\geq \frac{\pi}{2}$. To recover $u(T,x)$, we need to choose $p>\frac{\pi^2}{4}T$ , resulting in 
	\begin{equation*}
		\begin{aligned}
		u(T,x) &= e^p \int_{- \frac{\pi}{2}}^{\frac{\pi}{2}}  e^{-(p-\eta^2 T)} \hat u_0 e^{-ix\eta} d\eta =\int_{\bbR} e^{\eta^2 T} \hat{u}_0 e^{-ix\eta}d \eta,	 \quad p>\frac{\pi^2}{4}T.
		\end{aligned}
	\end{equation*}
	
	The computation stops at $T=100$, which represents a relatively long duration for an unstable system. 
    %Solving this without employing special regularization techniques is challenging. 
    In this work, we apply the Schr\"odingerization method to partial differential equations, followed by spatial discretization with $p\in \Omega_p=(-257, 257)$,  $\triangle x = \frac{1}{2^5}$, $\triangle p =\frac{257}{2^{9}}$ and $\triangle t = \frac{25}{2^{10}}$. This method fully exploits the exponential decay property in Fourier space.  The magnitude of  $p^{\Diamond}$ is $\mathcal{O}(T)$. For more details on selecting $p^{\Diamond}$ to recover the solution of the ill-posed system and the corresponding error analysis, we refer to \cite{JLM24SchrBackward}. The results are shown on the right side of Fig.~\ref{fig:recovery of u_illposed}, where the numerical solutions recovered either by selecting a single point or through numerical integration are close to the exact solution. The error $\|\Bu - \Bw_h^d(p)e^{p}\|$ with respect to $p$ is illustrated on the left side of Fig. \ref{fig:recovery of u_illposed}, demonstrating that by appropriately choosing the interval  $(p^{\Diamond}, p^{\Diamond}+R)$ with $e^R=\mathcal{O}(1)$ and $p^{\Diamond}=\pi^2T/4$, the target solution can be effectively recovered.

Classical methods with traditional regularization techniques, such as filtering or other approaches \cite{filterBHE, Scheme_BHE}, require careful parameter selection: too small leads to instability, too large results in excessive smoothing and loss of details. 
In contrast, 
%Schr\"odingerization guarantees exact recovery for continuous functions, as shown in Theorem \ref{thm:recovery u}. 
Schrödingerization, as guaranteed by Theorem \ref{thm:recovery u}, provides accurate (modulus truncation of high-frequency modes of the initial data) approximations with precise error estimates from the perspective of continuous functions.
While similar to Fourier regularization \cite{FourierBHE}, our method is simpler to implement. A detailed comparison is provided in \cite{JLM24SchrBackward}.
	
	\begin{figure}[t!]
		\includegraphics[width=0.45\linewidth]{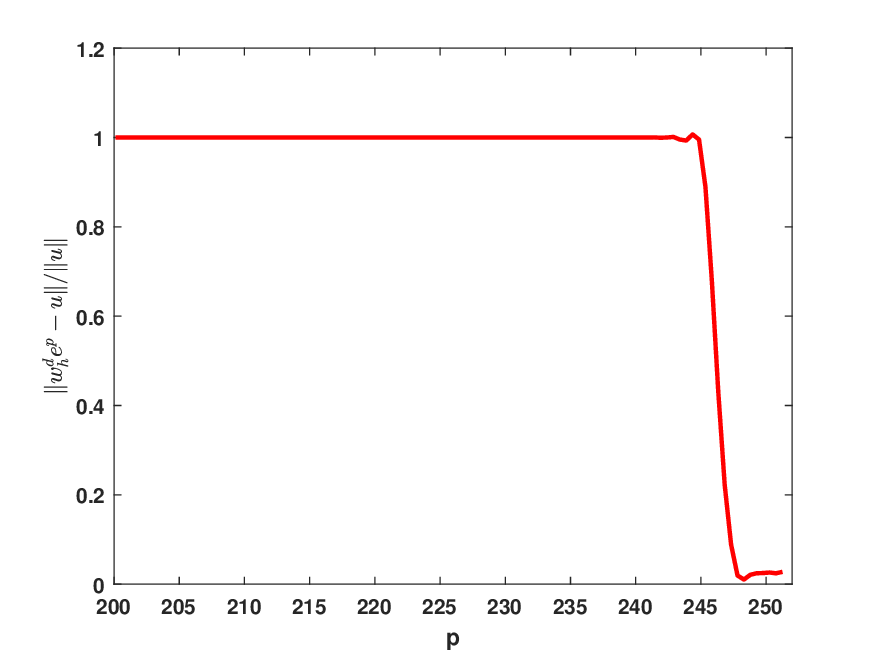}
		\includegraphics[width=0.45\linewidth]{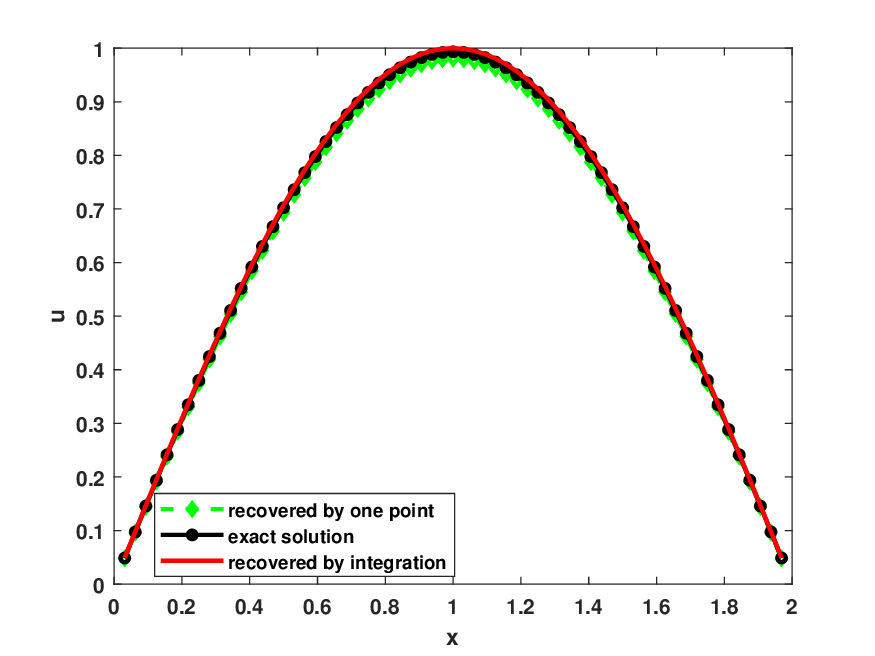}
		\caption{ Left: error of $\|\Bu - \Bw_h^d(p)e^{p}\|$
			with respect to $p$ with $\Bw_h^d$ defined in \eqref{eq:whd(t,x,p)}.  Right: the recovery from Schr\"odingerization by choosing $p>p^{\diamond} = \frac{\pi^2}{4}T\approx 247$.	
		}\label{fig:recovery of u_illposed}
	\end{figure}
	
	\subsection{Convergence rates of the discretization for Schr\"odingerized systems}
	
	In these tests, we  use the numerical simulation to test the convergence rates of the discretization of Schr\"odingerization and compare different discretization schemes for the Schr\"odingerized equations.
	
	First, we consider a 1-D case for Maxwell's equations.
	For the three-dimensional case, a similar approach can be adopted straightforwardly \cite{JLM23}. The electric field is assumed to have a transverse component $E_y$, i.e. $\BE =(0,E_y(x,t),0)$. The magnetic field is aligned with the z direction and its magnitude is denoted by $B_z$, i.e. $\BB=(0,0,B_z(x,t))$. The reduced Maxwell system with periodic boundary conditions is written as
	\begin{equation}\label{eq:maxwell tests}
		\partial_t E_y + \partial_x B_z =-J_y,\quad 
		\partial_t B_z + \partial_x E_y = 0,\quad \text{in}\;[0,1].
	\end{equation}
  The simulation stops at $T=1$.
	\subsubsection{An in-homogeneous system with source terms}
	\label{subsub1}
	The source term is defined as
    $$ J_y(t,x) = -2\pi t \cos(2\pi x).$$
	The initial conditions are prescribed as follows
	$$E_y(0,x) = \cos(2\pi x)/(2\pi)-1/(2\pi), \quad 
	B_z(0,x) = 0.$$
	The exact solution to the system is given by
	\begin{equation*}
		E_y = \cos(2\pi x)/(2\pi)-1/(2\pi),\quad 
		B_z = t\sin(2\pi x).
	\end{equation*}
	Yee's scheme \cite{TafYee00} for spatial dicretization gives
	\begin{equation*}
		\frac{\D}{\D t} \begin{bmatrix}
			\Bu\\
			\Br
		\end{bmatrix}
		=\begin{bmatrix}
			A &B\\
		O &O
		\end{bmatrix}\begin{bmatrix}
			\Bu\\
			\Br
		\end{bmatrix},
	\end{equation*}
	where $\Bu = \sum\limits_{i=0}^{n-2} E_{i+1}\ket{i} +\sum\limits_{i=0}^{n-1} B_{i+\frac 12}\ket{n-1+i}$ and $\Br = \sum\limits_{i=0}^{n-2} \ket{i}$. The matrix $A\in \bbR^{2n-1,2n-1}$ and $B\in \bbR^{2n-1,n-1}$ are defined by 
	\begin{equation*}
		A=\begin{bmatrix}
			O  &-D_x   \\
			D_x^{\top} &O 
		\end{bmatrix},\quad 
		B = \begin{bmatrix}
			-J_y \\
			\zerobf 
		\end{bmatrix}, \quad 
		D_{x} =  \frac{1}{h}
		\begin{bmatrix}
			1 &-1 & 0 &\cdots &0\\
			&1  &-1 &\cdots &0\\
			&   &\ddots &\ddots  &\\
			&   &       &1 &-1
		\end{bmatrix}\in \bbR^{n-1,n},
	\end{equation*}
	where $\BJ_y = \text{diag}\{\BJ_{h}\}$ with 
	$\BJ_h = \sum\limits_{i=1}^{n-1} J_y(x_{i},t)\ket{i}$.
	Then $\tilde H_1$ is obtained by
	$	\tilde H_1 = \begin{bmatrix}
		O &B/2\\
		B/2 &O
	\end{bmatrix}.$
	Obviously, one gets $\lambda(\tilde H_1) = \mathcal{O}(|b|/2)=\mathcal{O}(\max_{x,t}\{J_y\}/2)$.
	\begin{table}[t!]
		\centering
			\caption{
		The convergence rates of  $\frac{\|\Bu_h^d-\Bu\|_{\BL^2(\tilde{\Omega}_p)}}{\|\Bu\|_{\BL^2(\tilde{\Omega}_p)}}$ and    %$\|\Bu_h^d(p)-\Bu(p)\|_{\BL^2([\frac{1}{2},\frac{3}{2}])}/\|\Bu\|_{L^2([\frac{1}{2},\frac{3}{2}])}$ and
		 $\frac{\|\Bu_h^c-\Bu\|_{\BL^2(\tilde{\Omega}_p)}}{\|\Bu\|_{\BL^2(\tilde{\Omega}_p)}}$,  with $\gamma = 0.1$ and $\tilde{\Omega}_p=(\frac{1}{2},\frac{3}{2})$.
	}\label{tab:err of recovery dis conti}
		\begin{tabular}{ccccccc}
			%\hline 
			\midrule [2pt]
			$(\triangle p,\; \triangle t)$          & $(\frac{\pi}{2^7},\, \frac{1}{2^8})$ & order & $(\frac{\pi}{2^8},\, \frac{1}{2^9})$ & order & $(\frac{\pi}{2^{9}},\, \frac{1}{2^{10}})$ & order
			\\ %\hline
			\hline
			$\frac{\|\Bu_h^d(p)-\Bu(p)\|_{\BL^2(\tilde{\Omega}_p)}}{\|\Bu\|_{\BL^2(\tilde{\Omega}_p)}}$       & 	4.50e-05    &- &1.33e-05  & 1.75 & 3.28e-06 & 2.02  \\
			\midrule [2pt]
			$(X,\;\triangle t)$          & $(80,\,\frac{1}{2^8})$ & order & $(160,\,\frac{1}{2^9})$ & order & $(320,\frac{1}{2^{10}})$ & order\\ 
			\hline
			$\frac{\|\Bu_h^c(p)-\Bu(p)\|_{\BL^2(\tilde{\Omega}_p)}}{\|\Bu\|_{\BL^2(\tilde{\Omega}_p)}}$       & 1.74e-03   &- &	4.67e-04 &1.90  & 1.78e-04 &1.39
			\\
			\midrule [2pt]
		\end{tabular}            
	\end{table}
	
	In this test, we choose $\gamma=0.1$ and the modified smooth initial values \eqref{eq:smooth initial} to  test the convergence rates of the discretization for Schr\"odingerization. 
	For \eqref{eq:shro dis source}, we fix the $p$-domain within $\Omega_p=(-\pi,\pi)$.
	For \eqref{eq:check w source term}, we fix the step size of $\xi$ with $\triangle \xi = \frac{5}{2^3}$ and $\tilde \Omega_p =(\frac{1}{2},\frac{3}{2})$.
	Tab.~\ref{tab:err of recovery dis conti}  shows that the optimal convergence order $\|\Bu_h^d-\Bu\|_{\BL^2(\tilde{\Omega}_p)}\sim \triangle p^2$ and $\|\Bu_h^c-\Bu\|_{\BL^2(\tilde{\Omega}_p)} \sim X^{-1}$ are obtained, respectively.
	Fig.~\ref{fig:convergence rates} shows  $\|\Bw_h^ce^{p}-\Bu\|$ and $\|\Bw_h^de^{p}-\Bu\|$ with respect to $p$.
	It implies that it is better to pick the point $p>p^{\Diamond}$  near $p^{\Diamond}$, to avoid the error 
	being magnified exponentially for large $p$ when using a single point to recover $\Bu$ by $\Bu_h^d = e^p \Bw_h^d$ (or $\Bu_h^c = e^p \Bw_h^c$).
	According to the proof in Theorem~\ref{thm:err of conti}, the residual of $\Bw_h^c-\Bw$ contains $\int e^{i\xi p} d\xi$, leading 
	to the oscillation of $\|\Bw_h^c e^p-\Bu\|$ in terms of  $p$. Therefore, 
	it would be better to apply integration to recover $\Bu$ when using continuous Fourier transformation for Schr\"odingerization.
	
	\begin{figure}[http]
		\includegraphics[width=0.313\linewidth]{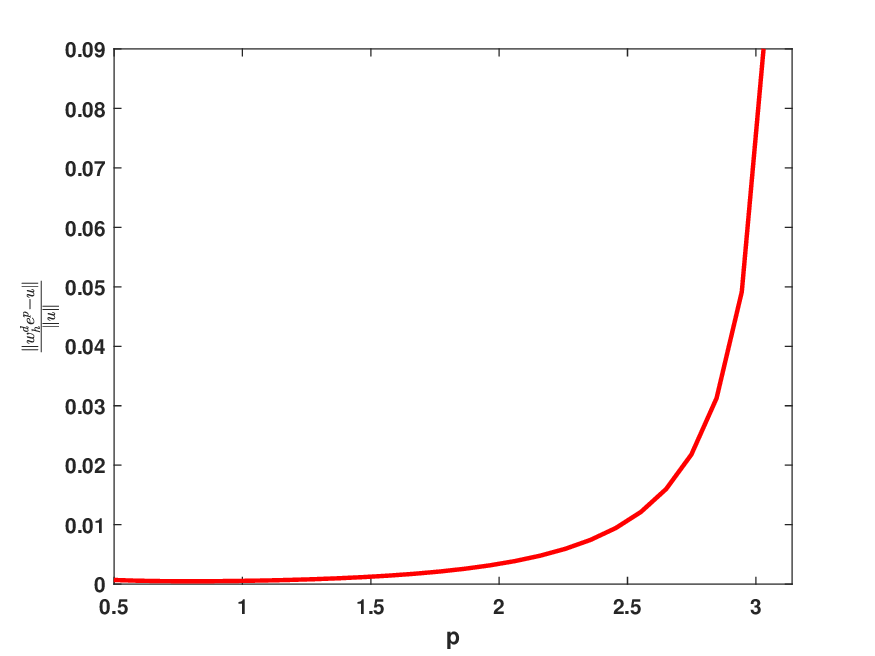}
		\includegraphics[width=0.313\linewidth]{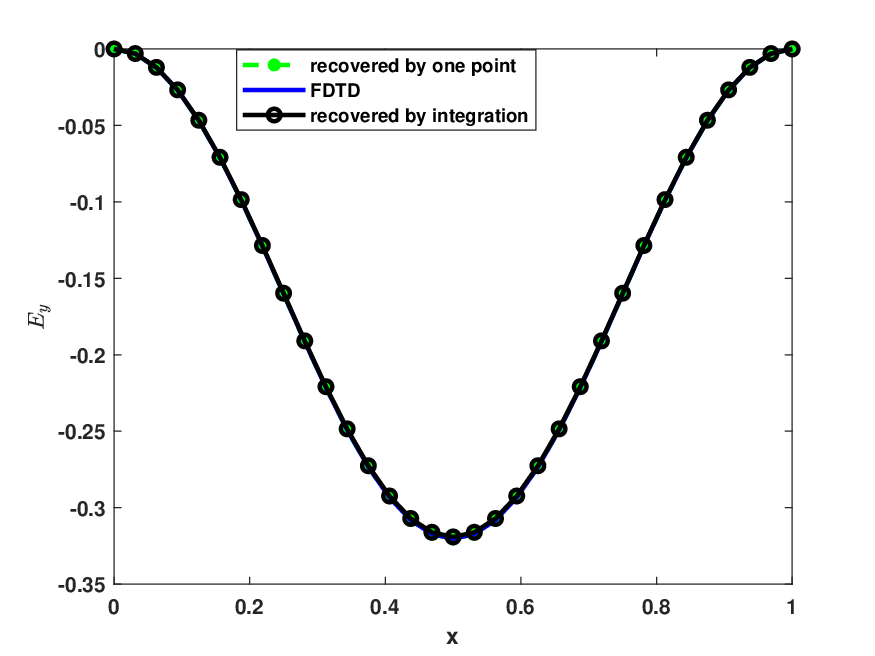}
		\includegraphics[width=0.313\linewidth]{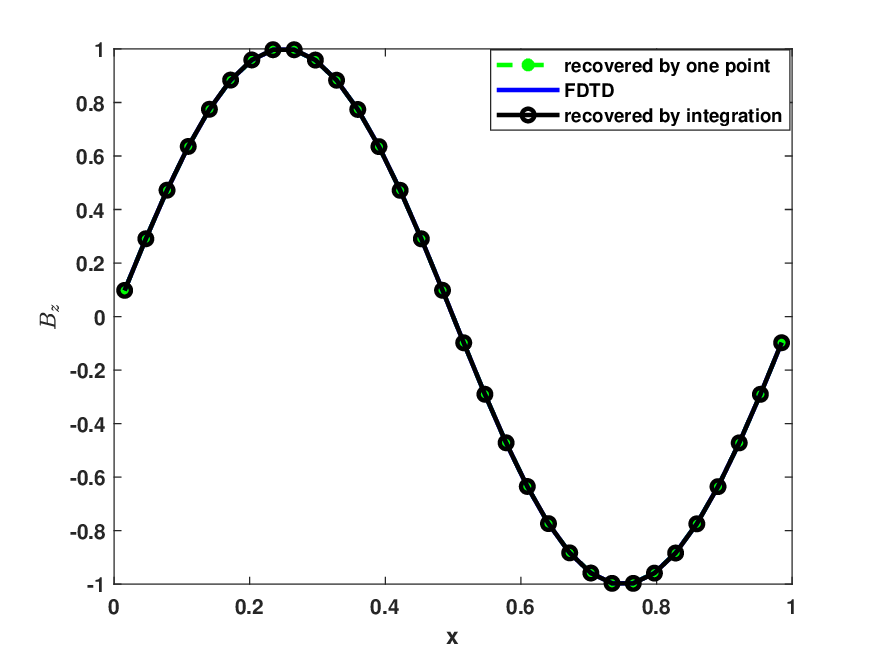}\\
		\includegraphics[width=0.313\linewidth]{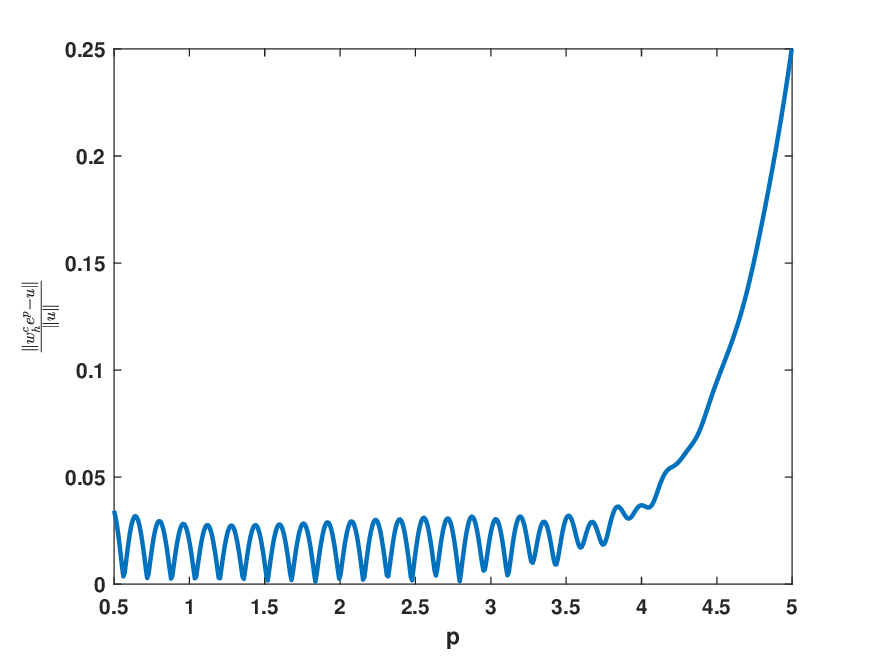}
		\includegraphics[width=0.313\linewidth]{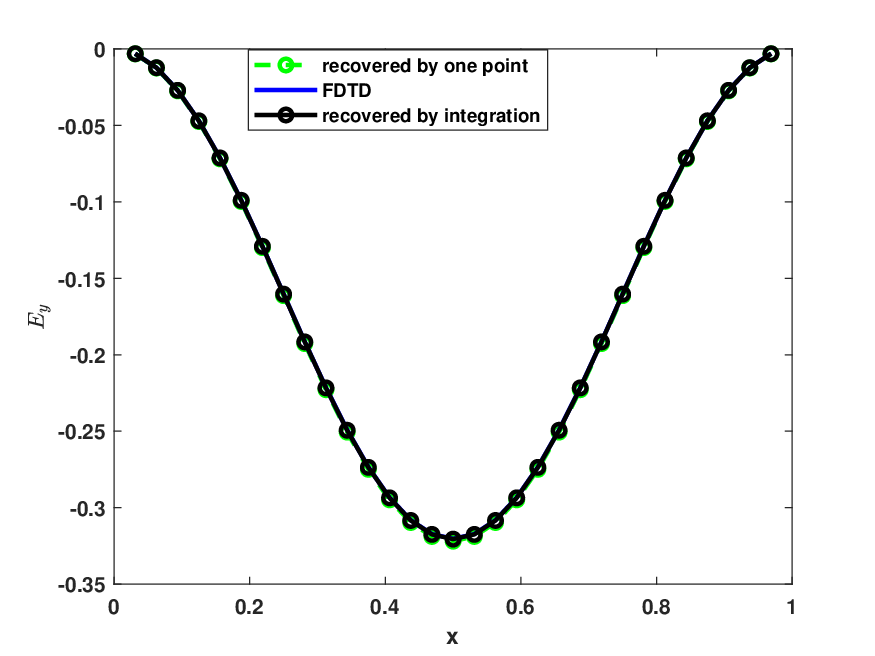}
		\includegraphics[width=0.313\linewidth]{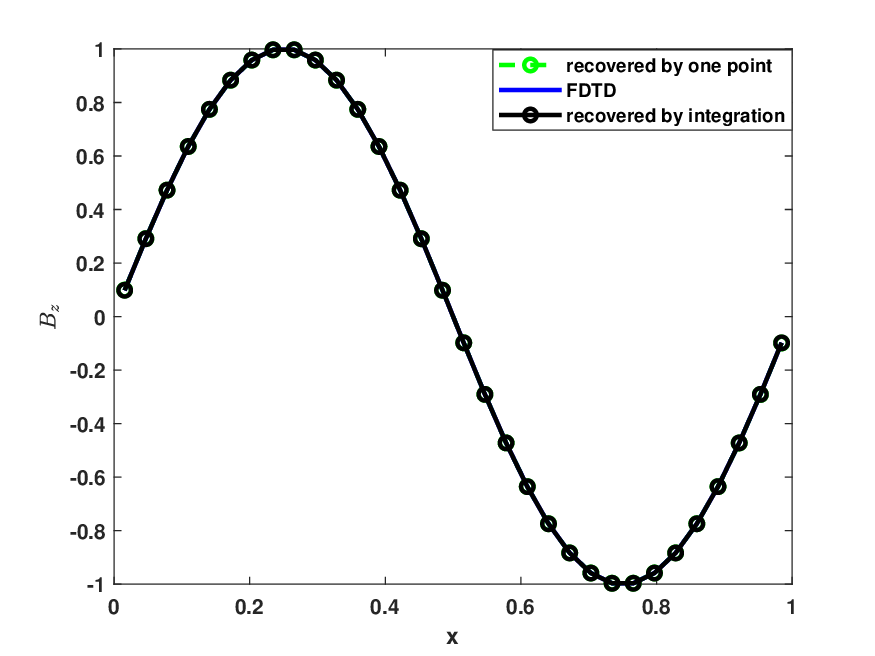}
		\caption{ The first row: the error of discrete Fourier transform defined by $\|\Bw_h^d e^{p}-\Bu\|/\|\Bu\|$ for Schr\"odingerization, with $\triangle p = \frac{\pi}{2^5}$ and $\triangle t = \frac{1}{2^5}$.
			The second row: the error of continuous Fourier transform defined by 
			$\|\Bw_h^c e^p-\Bu\|/\|\Bu\|$ for Schr\"odingerization, with $X = 160$ and $\triangle t = \frac{1}{2^5}$.}
		\label{fig:convergence rates}
	\end{figure}

	\subsubsection{An in-homogeneous system with big source terms}
	\label{subsub2}
	In this test, we fix the initial conditions of \eqref{eq:maxwell tests} and modify the source term  as
	\begin{equation*}
		J_y = -2000 \pi t\cos(2\pi x).
	\end{equation*}
	It is straightforward to observe that $\lambda(\tilde H_1) = \mathcal{O}(10^3)$, making it challenging to recover without employing the stretch transformation.
	By choosing $\gamma=10^{-4}$, the recovery from Schr\"odingerization agrees well with the  exact solution (see Fig.~\ref{fig:approximation of schro e1-4}).
	Comparing Tab.~\ref{tab:err of recovery dis conti} and Tab. \ref{tab:err of recovery dis conti bigSource}, we find that the stretch coefficient
	does not affect the relative error and the convergence order.
	
		\begin{table}[t!]
			\centering
		\caption{
			The convergence rates of  $\frac{\|\Bu_h^d-\Bu\|_{\BL^2(\tilde{\Omega}_p)}}{\|\Bu\|_{\BL^2(\tilde{\Omega}_p)}}$ and    %$\|\Bu_h^d(p)-\Bu(p)\|_{\BL^2([\frac{1}{2},\frac{3}{2}])}/\|\Bu\|_{L^2([\frac{1}{2},\frac{3}{2}])}$ and
			$\frac{\|\Bu_h^c-\Bu\|_{\BL^2(\tilde{\Omega}_p)}}{\|\Bu\|_{\BL^2(\tilde{\Omega}_p)}}$,  with $\gamma = \frac{1}{10^4}$ and $\tilde{\Omega}_p=(\frac{1}{2},\frac{3}{2})$.
		}\label{tab:err of recovery dis conti bigSource}
		\begin{tabular}{ccccccc}
			%\hline 
			\midrule [2pt]
			$(\triangle p,\; \triangle t)$          & $(\frac{\pi}{2^8},\, \frac{1}{2^7})$ & order & $(\frac{\pi}{2^9},\, \frac{1}{2^8})$ & order & $(\frac{\pi}{2^{9}},\, \frac{1}{2^{10}})$ & order
			\\ %\hline
			\hline
			$\frac{\|\Bu_h^d(p)-\Bu(p)\|_{\BL^2(\tilde{\Omega}_p)}}{\|\Bu\|_{\BL^2(\tilde{\Omega}_p)}}$       & 	4.09e-05    &- &1.21e-05  & 1.75 & 2.99e-06 & 2.02  \\
			\midrule [2pt]
			$(X,\;\triangle t)$          & $(80,\,\frac{1}{2^8})$ & order & $(160,\,\frac{1}{2^9})$ & order & $(320,\frac{1}{2^{10}})$ & order\\ 
			\hline
			$\frac{\|\Bu_h^c(p)-\Bu(p)\|_{\BL^2(\tilde{\Omega}_p)}}{\|\Bu\|_{\BL^2(\tilde{\Omega}_p)}}$       &1.58e-03   &- &	4.14e-04 &1.93  & 1.32e-04 &1.64
			\\
			\midrule [2pt]
		\end{tabular}            
	\end{table}

	\begin{figure}[http]
		\includegraphics[width=0.313\linewidth]{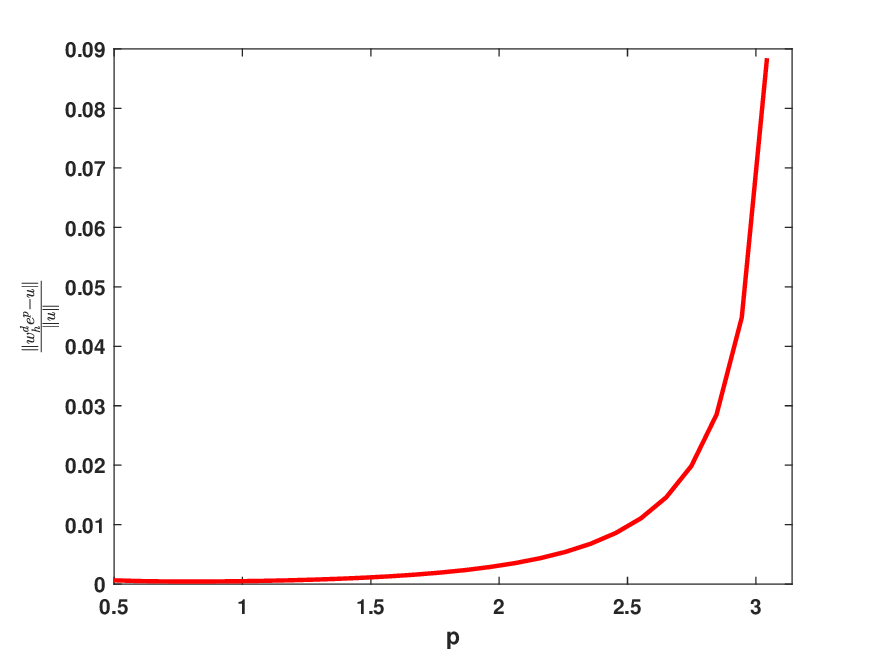}
		\includegraphics[width=0.313\linewidth]{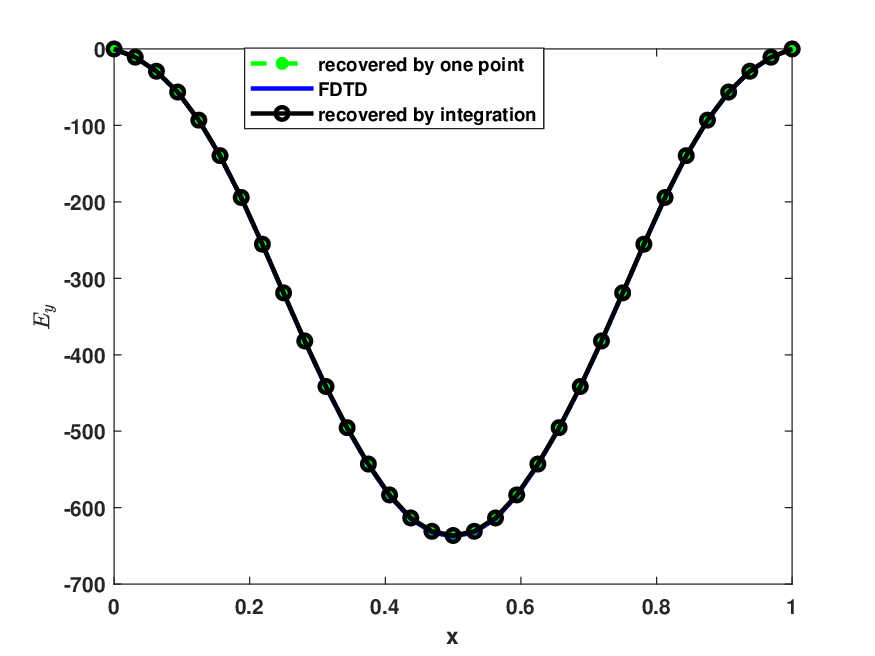}
		\includegraphics[width=0.313\linewidth]{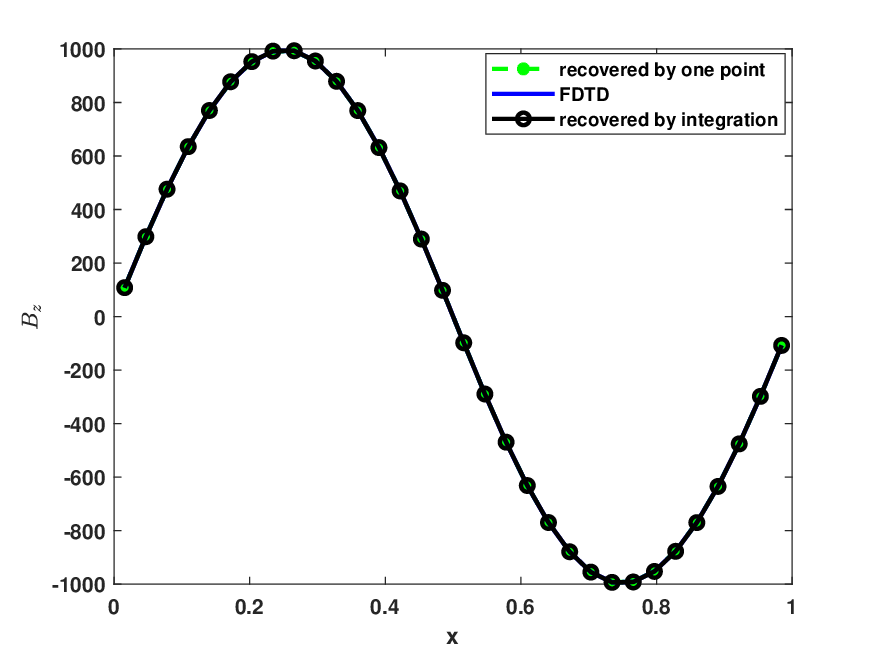}\\
		\includegraphics[width=0.313\linewidth]{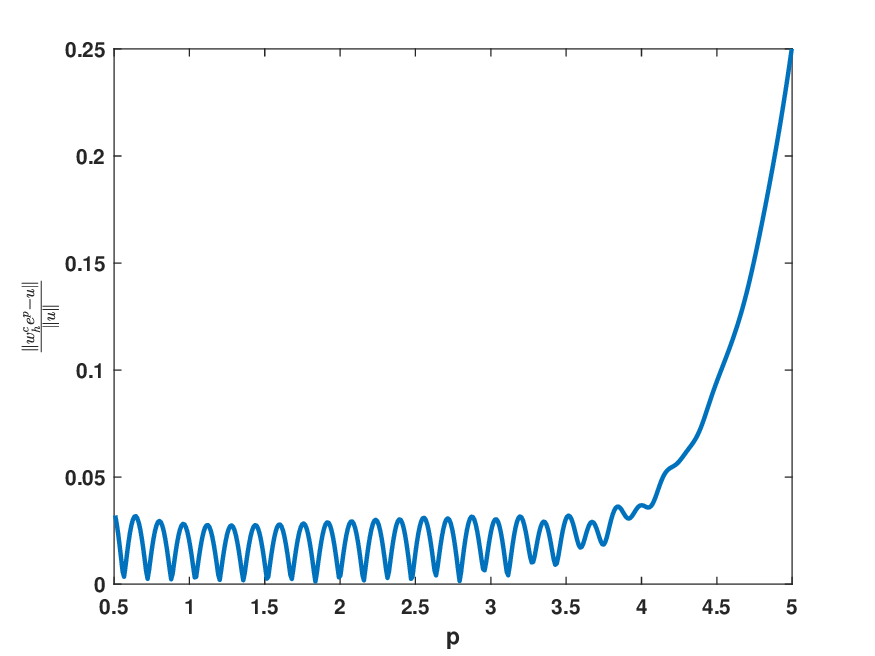}
		\includegraphics[width=0.313\linewidth]{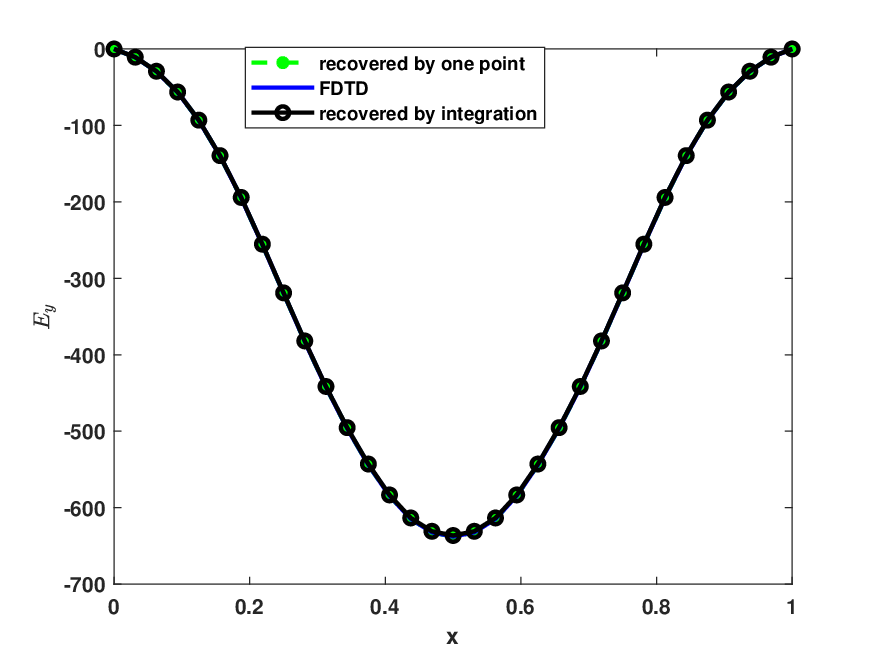}
		\includegraphics[width=0.313\linewidth]{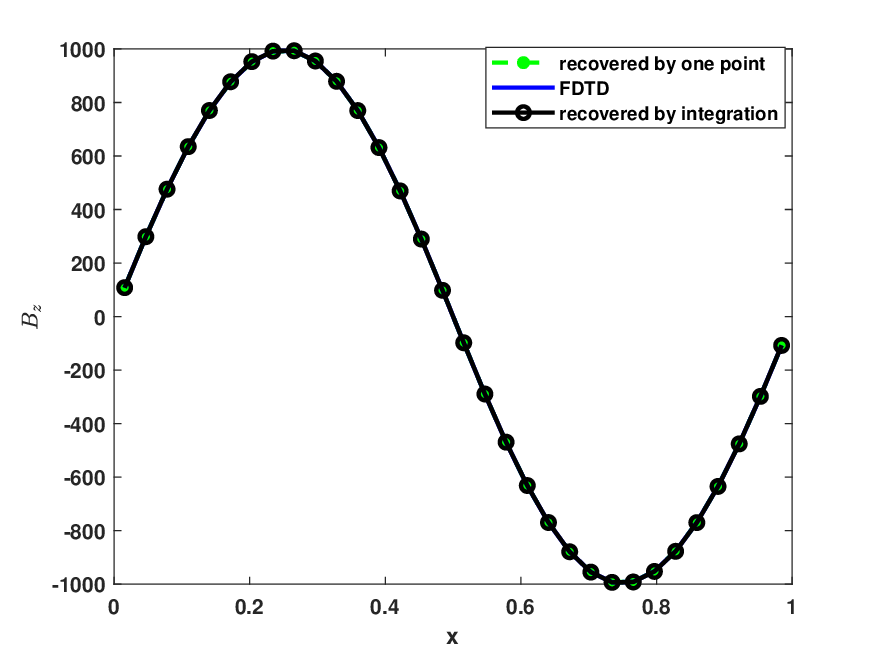}
		\caption{ The first row: the error of discrete Fourier transform defined by $\|\Bw_h^d e^{p}-\Bu\|$ for Schr\"odingerization, with $\triangle p = \frac{\pi}{2^5}$ and $\triangle t = \frac{1}{2^5}$.
			The second row: the error of continuous Fourier transform defined by 
			$\|\Bw_h^c e^p-\Bu\|$ for Schr\"odingerization, with $X = 160$ and $\triangle t = \frac{1}{2^5}$.}
		\label{fig:approximation of schro e1-4}
	\end{figure}
	
	By comparing the results of section \ref{subsub1} and section \ref{subsub2},
	we make the following summary.
	\begin{itemize}
		\item \textbf{Compare two  discretizations of Schr\"odingerization : \eqref{eq:schro_dis} and \eqref{eq:schro_conti_computer}.}
		When $\lambda_{\max}^-(H_1)$ is not particularly large, we refer to \eqref{eq:schro_dis} to discretize the Schr\"odingerized equations. Compared with \eqref{eq:schro_conti_computer}, the discrete Fourier transform 
		for Schr\"odingerization \eqref{eq:schro_dis} has a smaller error and higher-order convergence rates from Tab.~\ref{tab:err of recovery dis conti}.
		However, large $\lambda_{\max}^-(H_1)$ requires a particularly large  $p$ domain for  the discrete Fourier transform and very small $\triangle \xi$ for the continuous Fourier transform \eqref{eq:schro_conti_computer}.
		\item \textbf{Compare two recovery methods: \eqref{eq:recover by one point} and \eqref{eq:recover by quad}.}
		According to Fig.~\ref{fig:convergence rates} and Fig. \ref{fig:approximation of schro e1-4}, the recovery of the primitive variables via a single point (i.e. \eqref{eq:recover by one point}) is suitable for the discrete Fourier transform \eqref{eq:schro_dis},
		and the recovery of the solution through integration (i.e. \eqref{eq:recover by quad}) is suitable for the continuous Fourier transform \eqref{eq:schro_conti_computer}, due to the oscillations of the error with $p$.
	\end{itemize}

	\section{Conclusions}
	
	In this paper, we present some analysis and numerical investigations of the Schr\"odingerization of a general linear system with a source term.
	Conditions under which to recover the original variables are studied theoretically and numerically, for systems that can contain unstable modes.  
	We give the implementation details of the discretization of the Schr\"odingerization and prove the corresponding error estimates and convergence orders. 
	In addition, we homogenize the source term by using the stretch transformation and show that the stretch coefficient does not affect the error estimate of the quantum simulation. 
	
	As can be seen from the analysis, it is difficult to recover the target variables when the evolution matrix has large positive eigenvalues, as in chaotic systems.
	Our method provides a simple and general way to construct stable (including both classical and quantum) computational methods for unstable, ill-posed  problems. This will be the subject of  our further investigation \cite{JLM24SchrBackward}. 
	
	\section*{Acknowledgement}
	
	SJ and NL  are supported by NSFC grant No. 12341104,  the Shanghai Jiao Tong University 2030 Initiative and the Fundamental Research
	Funds for the Central Universities. SJ was also partially supported by the NSFC grants Nos. 12350710181 and 12426637, the Shanghai Municipal Science and Technology Major Project (2021SHZDZX0102).  NL also  acknowledges funding from the Science and Technology Program of Shanghai, China (21JC1402900). 
	CM was partially supported by China Postdoctoral Science Foundation (No. 2023M732248) and Postdoctoral Innovative Talents Support Program (No. BX20230219).

%%%%%%%%%%%%%%%%%%%%%%%%%%%%%%%%%%%%%%%%%%%%%%%%%%%%%%%%%%%%%%%%%%%
	%\bibliographystyle{siamplain}
	%\bibliography{references_inhorm}

\end{document}